\documentclass[11pt,a4paper]{article}
\usepackage[utf8]{inputenc}
\usepackage[T1]{fontenc}

\usepackage{amsmath}
\usepackage{amsthm}
\usepackage{amssymb}
\usepackage{amsopn}
\usepackage{enumitem}
\usepackage{color}
\usepackage{mathtools}
\usepackage{hyperref}
\usepackage{doi}
\usepackage[normalem]{ulem}
\usepackage{tikz}

\theoremstyle{plain}
\newtheorem{theorem}{Theorem}[section]
\newtheorem{proposition}[theorem]{Proposition}

\newtheorem{lemma}[theorem]{Lemma}

\theoremstyle{definition}
\newtheorem{definition}[theorem]{Definition}

\theoremstyle{remark}
\newtheorem{remark}[theorem]{Remark}

\DeclareMathOperator{\divo}{div}
\DeclareMathOperator{\Div}{Div}
\DeclareMathOperator{\Id}{Id}
\DeclareMathOperator{\sym}{sym}

\DeclareMathOperator{\spano}{span}
\DeclareMathOperator{\supp}{supp}

\newcommand{\G}{\mathcal{G}}

\def\<{\langle}
\def\>{\rangle}

\newcommand{\R}{\mathbb{R}}

\newcommand{\Rc}{\ensuremath{\mathcal{R}}}

\newcommand{\La}{\Lambda}

\def\HH{\mathcal{H}^{1}}
\def\HHc{\mathcal{H}^{c}}
\def\Ladef{\La^{\rm def}}
\def\Rdef{R^{\rm def}}

\def\CC{{\bf C}}

\newcommand{\Z}{\mathbb{Z}}
\newcommand{\C}{\mathbb{C}}
\newcommand{\N}{\mathbb{N}}

\newcommand{\eps}{\varepsilon}

\newcommand{\abs}[1]{\lvert #1 \rvert_0}

\usepackage{color}

\begin{document}

\begin{center}
\begin{Large}
Asymptotic Expansion of the Elastic Far-Field of a Crystalline Defect\footnote{JB and CO are supported by EPSRC Grant EP/R043612/1.}
\end{Large}
\\[0.5cm]
\begin{large}
Julian Braun\footnote{Mathematics Institute, University of Warwick, Coventry, CV4 7AL, UK.}\\
Thomas Hudson\footnotemark[2]\\
Christoph Ortner\footnote{Department of Mathematics, University of British Columbia, 1984 Mathematics Road, Vancouver, BC, V6T 1Z2, Canada.}
\end{large}
\\[0.5cm]
\today
\\[1cm]
\end{center}

\begin{abstract}
Lattice defects in crystalline materials create long-range elastic fields which can be modelled on the atomistic scale using an infinite system of discrete nonlinear force balance equations. Starting with these equations, this work rigorously derives a novel far-field expansion of these fields: The expansion is computable and is expressed as a sum of continuum correctors and discrete multipole terms which decay with increasing algebraic rate as the order of the expansion increases. Truncating the expansion leaves a remainder describing the defect core structure, which is localised in the sense that it decays with an algebraic rate corresponding to the order at which the truncation occurred.
\end{abstract}





\section{Introduction}
The field of continuum solid mechanics has been highly successful in providing robust predictions of material behaviour at a wide range of length-scales. 
In crystalline materials in particular, it is recognised that the predictions made by the equations of linear elasticity are valid with tolerable errors even when resolving features such as defects whose characteristic size is close to that of the interatomic spacing. 
This said, when considering processes which involve the genuine thermodynamic, electronic and chemical properties of such defects, the fundamental discreteness of matter becomes crucial and no single continuum model can be sufficient to completely capture the fine detail of a material's behaviour at this scale. 
Moreover, it is exactly these fine details which determine phenomena such as a material's yield strength and behaviour under cyclic loading, both of which are crucial to understand for engineering applications. 

As a result, a range of theoretical techniques have been developed over the last 60 years which seek to predict and compute defect behaviour in crystals, connecting discrete and continuum models of these materials. 
Broadly, these approaches can be divided into two categories, namely \emph{concurrent} and \emph{sequential} modelling strategies. Models in the former class combine discrete and continuum models into a single system which can then be numerically solved simultaneously, while those in the latter class generally involve iteration over separate models acting at different scales. 
In particular, the last 25 years has seen a great deal of research activity focused on concurrent strategies, with one of the most significant developments in this area being the quasicontinuum method \cite{TOP96} and many variants thereof \cite{LO13}.
In contrast, the present work revisits sequential strategies for accurately modelling defects, but with a new perspective that recent progress in the study of multi-scale models has enabled, and our results on the structure of crystalline defects have direct consequences for the design and evaluation of more general models, including both purely atomistic and concurrently coupled models. 

Starting from a discrete energy for a material defined on an infinite domain~\cite{EOS2016}, we develop a hierarchy of linear continuum PDE systems which can be efficiently derived and solved numerically. 
The solutions to these PDE systems form a sequence of smooth \emph{predictors} which describe the far-field behaviour of the lattice strain around a defect to within arbitrary accuracy, and thus provide increasingly accurate boundary conditions to be used on the discrete model when confined to a finite domain. 
The key idea behind our approach is to exploit the knowledge that the variation in the strain field generally decays smoothly away from the core of any localised defect, and hence the far-field behaviour is more and more accurately predicted by the solutions of the continuum PDE models we define. 
These far-field approximate solutions are coupled with the properties of the discrete defect core, encapsulated by the spatial moments of the acting forces and expressed as a multipole expansion. The relative simplicity of these moments provides an elegant, computable way to transfer information from the nonlinear discrete problem to the continuum hierarchy. The coupling moreover is ``weak'' in the sense that a term in the continuum hierarchy only requires information on the multipole terms of strictly lower order. Indeed all terms are defined and are computable in sequence without concurrent coupling.

Our approach has connections with classical approaches to modelling defects in continuum linear elasticity using the defect dipole tensor \cite{E56,NH63} (also known as the elastic dipole tensor or the double force tensor) and to Sinclair's work on atomistic models of fracture \cite{SL72,S75}. 
More recent related work includes that of Trinkle and coworkers, where lattice Green's functions have been used to improve the accuracy of defect computations \cite{T08,TT16}, along with mathematical developments in our understanding of the regularity of discrete strain fields advanced by the authors and coworkers \cite{HO2012,EOS2016,braun16static,BBO19}. In particular, \cite{BBO19} explores some of the initial ideas of our mathematical strategy in a simplified setting. The approach presented here unifies many of the ideas involved in these previous works into a single framework and expands them systematically to higher orders.

More generally, the powerful structural results we present here serve as a useful tool in any discussion of crystalline defects where high accuracy is required. In particular, we will outline how our results may be used as a foundation for a rigorous numerical analysis of defect algorithms and also provide a path to systematically improve their accuracy.

\subsection{Methodology}
Our starting point is to consider a total energy $\mathcal{E}(u)$ for displacements $u$ of an infinite lattice $\La$ of atoms. We will make the mild assumption that the total energy of a displacement is expressed as a sum of \emph{site energies}, i.e. contributions to the total energy arising from the environment of each atom. Under the assumption of frame indifference, this energy (and the site energies which make up the total energy) must depend only upon the relative displacements between atoms. 

An equilibrium displacement $\bar{u}$ of the energy $\mathcal{E}$ satisfies the force balance equation
\begin{equation} \label{eq:forceequilibirum}
	\delta \mathcal{E}(\bar{u}) = 0,
\end{equation}
and it is this infinite system of discrete equations which we study. Since equilibrium displacements $\bar{u}$ exhibit decay properties away from the core of point defects and dislocations \cite{EOS2016}, we can develop these equations around the state of zero displacement to derive approximate equations for the equilibrium displacement $\bar{u}$ in the far field. This comes in two steps:
\begin{itemize}
\item As a first step, we can expand the site potential around zero to obtain linear and nonlinear lattice operators, which still depend on finite differences (atomic bonds); and
\item As a second step, we can use an expansion for the displacement itself to replace finite differences with a gradient and higher derivatives, obtaining continuum PDE approximations to the discrete lattice equations.
\end{itemize}
Note already, that the latter Taylor expansion requires sufficiently smooth continuum displacements and cannot be applied to the discrete $\bar{u}$ itself, so some care needs to be taken when applying this strategy. The PDE approximations come in the form of a hierarchy of corrector equations. 
Crucially, each one of these corrector equations has the same form: the continuum linear elasticity (CLE) equation must be solved but with different right-hand sides (forces) that depend on previous terms in the expansion. 
Such hierarchies of corrector equations have been explored by other authors before, e.g., see \cite{emingstatic} for a formal hierarchy of similar corrector equations in the defect-free setting. In the present work, we take great care to define all terms rigorously, ensure sufficient regularity, and provide sharp estimates for all resulting error terms.

We then combine these continuum correctors with a multipole series which can be obtained from the moments of linearised residual forces. Overall, this enables us to obtain far-field approximations to arbitrarily high order for $\bar{u}$, which is characterised by a discrete remainder whose locality (decay) is precisely controlled.

While our general approach and many of our theoretical results are generic, some significant technical and conceptual challenges come into play when applying the results to specific defects. Specifically, questions regarding the geometry and the relation of a defect state to a reference lattice, as well as the precise properties of the lattice Green's function after adjusting for the geometry. These challenges can sometimes be overcome adhoc in leading order (as in \cite{EOS2016} for edge dislocations) but require a more detailed understanding for higher orders. Our main focus here is the general methodology. We will therefore restrict ourselves to two defect types that are geometrically relatively simple. Namely, general point defects and (straight) screw dislocations.

\subsection{Outline}
 The paper is organised as follows: In \S~\ref{sec:genresults}, we develop our general theory independent of specific defects. We present our main theoretical results concerning the decay of displacement fields and their approximation by multipole expansions for linearised lattice models in \S~\ref{sec:results:general} and then outline the derivation of the continuum corrector PDEs in \S~\ref{sec:contdev}. 
 
In \S~\ref{sec:applications}, we apply our methodology to our nonlinear atomistic models of point defects in Theorem \ref{thm:pointdef} and screw dislocations in Theorem \ref{thm:screw}, and explain the implications for the convergence of numerical methods which exploit the result. \S~\ref{sec:conclusions} presents our conclusions, and discusses the outlook for extending and applying these results in the future.

The proofs of our decay estimates for the linearised models are then provided in \S~\ref{sec:proofs-decay}, a full discussion of the lattice Green's function in  \S~\ref{sec:greensfunctions}, and the proofs of Theorem \ref{thm:pointdef} and \ref{thm:screw} and in \S~\ref{sec:prooffarfieldexpansion}.

\section{General Results} \label{sec:genresults}

\subsection{Models and Notation} \label{sec:results:model}

\paragraph{The Atomistic Energy and General Notation.}
Our results are concerned with modelling of crystalline defects, i.e. local regions of non-uniform atom arrangements embedded in a homogeneous host crystal. In this section we will start with the homogeneous setting itself. In \S~\ref{sec:applications} we will then look at the description of defect configurations and discuss how the homogeneous results can be applied there.

The homogeneous crystal is described by a Bravais lattice $\La := A \Z^d$ as the reference, where $d \geq 2$ and $A \in \R^{d \times d}$ is non-singular, and displacements $u \colon \La \to \R^N$, where we allow $N \neq d$ in order to model a range of scenarios (for example, in some models for pure screw dislocations, we have $d = 2$, $N = 3$, while in anti-plane shear $d = 2$, $N = 1$).

We denote discrete differences by $D_\rho u (\ell) := u(\ell + \rho) - u(\ell)$ for $\ell, \rho \in \La$. Later on we will also look at higher discrete differences which we denote by $D_{\boldsymbol{\rho}} = D_{\rho_1} ... D_{\rho_j}$ for a $\boldsymbol{\rho} = (\rho_1, ..., \rho_j) \in \La^j$.

Next we look at an interaction neighbourhood $\Rc \subset \La \backslash \{ 0\}$. We assume throughout that $\Rc$ is {\em finite}, $\Rc= -\Rc$, and that it spans the lattice $\spano_\Z \Rc= \La$. Based on $\Rc$ we define the discrete difference stencil $Du(\ell) := D_{\Rc} u(\ell) := (D_\rho u (\ell))_{\rho \in \Rc}$. Again, later we will consider higher discrete differences and in particular apply $D$ $k$ times for which we use the simple notation $D^k u = D \dots D u$. With the discrete difference stencil we can formally write down the energy of $u$ as
\[\mathcal{E}(u) = \sum_{\ell \in \La} V(Du(\ell)),\]
where $V \colon \R^{N \times \Rc} \to \R$ is the site energy. We will assume throughout that $V \in C^{K}(\R^{N \times\Rc})$ for some $K$ and satisfies the natural and very mild symmetry assumption 
\begin{equation} \label{eq:pointsymmetry}
V(A) = V((-A_{-\rho})_{\rho \in \Rc}).
\end{equation}

 Note that this is only a formal definition of the energy as the sum might not converge. All these quantities might also need to be adjusted to inhomogeneous generalisations $\mathcal{E}^{\rm def}$, $\Ladef$, $\Rc_\ell$, and $V_\ell$ to allow for a desired defect structure. Both of these aspects will be discussed in detail in \S~\ref{sec:applications}. In particular, we will make $\mathcal{E}$ precise and establish differentiability properties for any $u$ in the discrete energy space $\HH$ which is defined by
\[
  \HH=\HH(\La) = 
  \big\{ 
    u : \La \to \R^N \,|\, \|u\|_{\HH}:=\|Du\|_{\ell^2} < \infty
  \big\}.
\]
For future reference we also define the dense subspace $\HHc \subset \HH$ 
of displacements with compact support, 
\[
    \HHc = \HHc(\La) = 
    \big\{ 
      u : \La \to \R^N \,|\, {\rm supp}(Du) \text{~bounded~}
    \big\}.
\]
Variations of $\mathcal{E}$ can then be written as
\[\delta^k \mathcal{E} (u)[v_1, \dots, v_k] = \sum_{\ell \in \La} \nabla^k V(Du)[Dv_1, \dots, Dv_k].\]
More generally, we use the notation $T[a_1, \dots, a_k]$ for a multi-linear operator $T$. For symmetric operators we will shorten the notation further and write $T[a]^k:= T[a,\dots,a]$ if $a_1=\dots = a_k = a$ or $T[a]^m[b]^{(k-m)}$ if $m$ inputs are $a$ and $k-m$ are $b$. We denote the Hessian at zero by
\[H[u,v] = \delta^2 \mathcal{E} (0)[u,v]  = \sum_{\ell \in \La} \nabla^2 V(0)[Du,Dv]. \]

It will be convenient to interpret these objects as linear functionals, belonging to $(\HH)^\ast$,  acting on the last test function. We will often use a point wise representation based on the $\ell^2$ scalar product. For example,
\[
\delta^k \mathcal{E} (u)[v_1, \dots, v_{k-1}](\ell) := -\Div \big( \nabla^k V(Du)[Dv_1, \dots, Dv_{k-1}] \big),\]
and specifically
\begin{equation} \label{eq:Hdefinition}
 H[u](\ell) := -\Div \big( \nabla^2 V(0)[Du]\big),
\end{equation}
where $\Div A = -\sum_{\rho \in \Rc} D_{-\rho} A_{\cdot \rho}$ is the discrete divergence for a matrix field $A \colon \La \to \R^{N\times \Rc}$.

In this notation we can also write down the force equilibrium equations $\delta \mathcal{E}(u)=0$ in the pointwise form
\[0= \delta \mathcal{E}(u)(\ell) = -\Div \big(  \nabla V(Du(\ell))\big). \]

\paragraph{Lattice Stability and The Green's Function}\

We assume throughout that the Hamiltonian $H=\delta^2 \mathcal{E}(0)$ is {\em stable} and will equivalently call the lattice stable (see \cite{HO2012}), which by definition is the case if and only if there exists a $c_0 > 0$ such that  
\begin{equation} \label{eq:results:latticestab}
    H[u, u] \geq c_0 \| u \|_{\HH}^2 \qquad 
    \forall u \in \HHc.
\end{equation}
For stable operators $H$ there exists a {\em lattice Green's function} $\G : \La \to \R^{N \times N}$ such that $H[\G e_k](\ell) = e_k \delta_{\ell,0}$ for all $1 \leq k \leq N$ and such that
\begin{equation} \label{eq:decay-Gr}
    | D^j \G(\ell) | \lesssim \abs{\ell}^{-d-j+2} 
    \qquad \text{for } \ell \in \La,\  j \geq 1,
\end{equation}
see \cite{EOS2016}. Here we used the notation $\abs{\ell} := \lvert \ell \rvert + 2$ to write decay rates in the discrete setting in a more compact form which we will do throughout. We will also often just write $\G_k := \G e_k$.

\paragraph{The Cauchy-Born Continuum Model}\
Our atomistic lattice model naturally gives rise to corresponding continuum model based on the Cauchy-Born rule. In the continuum setting, one has an energy of the form
\[\mathcal{E}^{\rm C}(u) = \int_{\R^d} W(\nabla u)\,dx\]
for a displacement $u \colon \R^d \to \R^N$. The energy density $W$ is given by the Cauchy-Born rule
\begin{equation} \label{eq:defn_W_cb}
    W(M) := \frac{1}{c_{\rm vol}} V((M\rho)_{\rho \in \Rc})
\end{equation}
for any $M \in \R^{N \times d}$, where $c_{\rm vol} = \lvert \det A \rvert > 0$ is the volume of a lattice cell.

We will later see in more detail the usefulness and limitations of the nonlinear continuum model for defect problems. We will also make heavy use of the linearised continuum problem for our corrector equations. These are given through the continuum Hamiltonian $H^{\rm C}= \delta^2 \mathcal{E}^{\rm C}(0)$. In our pointwise notation the equilibrium equations is
\[0=H^{\rm C}[u](x)\]
or
\begin{equation} \label{eq:CLE}
0= -\divo \big( \mathbb{C}[\nabla u(x)] \big),
\end{equation} 
where $\mathbb{C} = \nabla^2 W(0)$. These are the standard continuum linear elasticity (CLE) equations.

The lattice stability \eqref{eq:results:latticestab} in particular also implies the Legendre-Hadamard stability of $\mathbb{C}$ (see \cite{HO2012,braun16static}), so that \eqref{eq:CLE} is elliptic and allows for a continuum Green's function or fundamental solution $G^{\rm C} \colon \R^d \backslash \{0\} \to \R^{N \times N}$ that solves
\[H^{\rm C}[G^{\rm C} e_k] = e_k \delta_0\]
in the distributional sense for $1 \leq k \leq N$.

\paragraph{Notation for Tensors}
To work with various higher order tensor products throughout, we establish a precise but compact notation: Given a $k$-tuple of vectors in $\R^d$, $\boldsymbol{\sigma}=(\sigma^{(1)}, \dots, \sigma^{(k)} ) \in (\R^d)^k$ we denote their $k$-fold tensor product \[\boldsymbol{\sigma}^\otimes := \bigotimes_{m=1}^k \sigma^{(m)} := \sigma^{(1)} \otimes \cdots \otimes \sigma^{(k)}.\]
The vector space spanned by these tensor products is denoted $(\R^d)^{\otimes k}$, and it is easy to see this space is isomorphic to $\R^{d^k}$. We also write
\[v^{\otimes k} := v \otimes ... \otimes v \in (\R^d)^{\otimes k}\] when considering the $k$-fold tensor product of a single vector $v \in \R^d$.

Let $S_k$ denote the usual symmetric group of all permutations which act on the integers $\{1,\ldots,k\}$. This action can be extended to $k$-tuples and tensor products by defining \[
  \pi(\boldsymbol{\sigma}) := (\sigma^{(\pi(1))}, \dots, \sigma^{(\pi(k))} )\quad\text{and}\quad\pi(\boldsymbol{\sigma}^\otimes):=\sigma^{(\pi(1))}\otimes\cdots\otimes \sigma^{(\pi(k))}\]
for any $\pi \in S_k$ and $\boldsymbol{\sigma} \in (\R^d)^k$. For any $\boldsymbol{\sigma}\in(\R^d)^k$, we define the symmetric tensor product by \[\boldsymbol{\sigma}^\odot:=  \sigma^{(1)} \odot \cdots \odot \sigma^{(k)} := \sym \boldsymbol{\sigma}^\otimes := \frac{1}{k!} \sum_{\pi \in S_k} \pi(\boldsymbol{\sigma})^\otimes. \]
The space spanned by these symmetric tensors is then denoted by $(\R^d)^{\odot k}$, and is a vector subspace of $(\R^d)^{\otimes k}$.

The natural scalar product on $(\R^d)^{\otimes k}$ and $(\R^d)^{\odot k}$ is denoted by $A : B$ for $A,B \in (\R^d)^{\otimes k}$ and, as usual, is defined to be the linear extension of
 \[\boldsymbol{\sigma}^\otimes : \boldsymbol{\rho}^\otimes := \prod_{m=1}^k \sigma^{(m)} \cdot \rho^{(m)} .\]

In particular, for $u \colon \La \to \R^N$ we have $D^k u(\ell) \in \R^N \otimes (\R^{\Rc})^{\otimes k}$. For a second tensor $\CC \in (\R^{\Rc})^{\otimes k}$ given specifically as a sum $\CC = \sum_{\boldsymbol{\rho} \in \Rc^{k}} \CC_{\boldsymbol{\rho}} \boldsymbol{\rho}^{\otimes}$ we can then write $\CC \colon D^k u(\ell) = \sum_{\boldsymbol{\rho}} \CC_{\boldsymbol{\rho}} D_{\boldsymbol{\rho}} u(\ell) \in \R^N$.

\subsection{General Results for the Linearised Equation}
\label{sec:results:general}
At the most fundamental level, our results concern the characterisation of the far-field behaviour of lattice displacements $u : \La \to \R^N$ that are close to equilibrium in the far-field. More precisely, given a stable Hamiltonian $H$ as introduced in the previous section we will characterise the decay of a general lattice displacement $u$ provided that the (linearised) residual forces $f(\ell) := H[u](\ell)$ decay sufficiently rapidly as $|\ell| \to \infty$. In \S~\ref{sec:contdev} we will then show how to use these results for the linearised operator to obtain characterisations of the far-field behaviour of equilibrium displacements in our full nonlinear interaction model.

If $\ell \mapsto H[u](\ell) \otimes \ell^{\otimes j} \in \ell^1(\La)$ we define the $j$-th moment 
\begin{equation} \label{eq:results:defn_Ij}
  \mathcal{I}_j[u] = \sum_{\ell \in \La} H [u](\ell) \otimes \ell^{\otimes j}.
\end{equation}

With this definition we have the following result.

\begin{theorem} \label{thm:structure}
Assume $\lvert H [u](\ell) \rvert \lesssim \abs{\ell}^{-d-p} \log^\alpha \abs{\ell}$ with $p, \alpha \in \N_0$, as well as $\mathcal{I}_i=0$ for $i=0, \dots, p-1$. Then, for $j = 1, 2$,
\begin{equation} 
    \label{eq:thm_structure_mainresult}
   \lvert D^j u(\ell) \rvert \lesssim \abs{\ell}^{2-d-p-j} \log^{\alpha+1} \abs{\ell} 
\end{equation}

If $j \geq 3$, then equation \eqref{eq:thm_structure_mainresult} is still true under the additional assumption that $\lvert D^m H [u](\ell) \rvert \lesssim \abs{\ell}^{-d-p-m} \log^\alpha \abs{\ell}$ for $m=1,...,j-2$.
\end{theorem}

\begin{remark} \label{rem:alphaneg}
If $H[u]$ satisfies the assumptions for some $\alpha < -1$ instead of $\alpha \in \N_0$, then there is no logarithmic factor needed in the result, i.e., 
\begin{equation} 
  \label{eq:thm_structure_mainresult_nolog}
 \lvert D^j u(\ell) \rvert \lesssim \abs{\ell}^{2-d-p-j}.
\end{equation}
The same holds true for the Theorems \ref{thm:structurewithmoments} and \ref{thm:structurewithmomentscont} below, if $\alpha < -1$.
\end{remark}

\begin{remark} 
  When initially reading both this theorem and the theorems in \S\ref{sec:genresults} and \S\ref{sec:applications}, we suggest to ignore the logarithmic terms and focus only on the algebraic rates. However, the treatment of the logarithmic terms is an important aspect of both our theorems and proofs since they appear to be intrinsic to the expansion, and not due to suboptimal estimates.
\end{remark}

Suppose now that we have a general elastic field $u$ with non-vanishing moments \eqref{eq:results:defn_Ij} but still fast decay of the residual forces. Then, we can decompose $u$ into a truncated multipole expansion --- higher order derivatives of the lattice Green's function defined in \S~\ref{sec:results:model} --- corresponding to the non-vanishing moments and a {\em far-field remainder} that exhibits the improved decay established in Theorem~\ref{thm:structure}. 
This idea is made precise in the next result.

\begin{theorem} \label{thm:structurewithmoments}
Assume $\lvert H [u]  \rvert \lesssim \abs{\ell}^{-d-p} \log^\alpha \abs{\ell}$ with $p, \alpha \in \N_0$. Furthermore, let $\mathcal{S} \subset \La$ be linear independent with $\spano_\Z \mathcal{S}=\La$. Then there are coefficients $b^{(i,k)} \in (\R^{\mathcal{S}})^{\odot i}$, $i=0, \dots, p-1$, $k=1,\dots, N$, such that
\begin{equation} \label{eq:multipole+remainder}
    u = \sum_{i=0}^{p-1} \sum_{k=1}^N b^{(i,k)} : D_{\rm \mathcal{S}}^i \mathcal{G}_k + w 
\end{equation}
and the remainder decays as 
\begin{equation}\label{eq:structurewithmomentsremainder}
\lvert D^j w \rvert \lesssim \abs{\ell}^{2-d-p-j}  \log^{\alpha+1} \abs{\ell} 
\end{equation} 
for $j=1,2$.

If $j \geq 3$, then equation \eqref{eq:structurewithmomentsremainder} is still true under the additional assumption that $\lvert D^m H [u] \rvert \lesssim \abs{\ell}^{-d-p-m} \log^\alpha \abs{\ell}$ for all $m=1,...,j-2$.
\end{theorem}

It is convenient to have a continuum reformulation of this multipole expansion, to avoid having to work with the discrete Green's function and its discrete derivatives. Towards that end we exploit the connections between continuum and discrete Green's functions and derive higher order continuum approximations of the discrete Green's function.
Specifically, in \S~\ref{sec:greensfunctions} we will construct a sequence of continuum kernels with the following properties.

\begin{theorem} \label{thm:greensfunctionexpansion}
There are unique kernels $G_n \in C^{\infty}(\R^d\backslash\{0\}; \R^{N \times N})$ such that
\[\Big\lvert D^j \G(\ell) - \sum_{n=0}^p D^j G_n(\ell)\Big\rvert \leq C_{j,p} \abs{\ell}^{-2p-d-j} \]
for all $\ell \in \La$ and $j,p \in \N_0$ and such that the $G_n$ are positively homogeneous of degree $(2-2n-d)$ if $n \geq 1$ or $d \geq 3$, while in the case $n=0, d=2$ we have
$G_0(\ell) = A \log \lvert \ell \rvert + \varphi(\ell)$, where $A \in \R^{N \times N}$ and $\varphi$ is $0$-homogeneous. Furthermore, we find $G_0 = G^{\rm C}$.
\end{theorem}

The $G_n$, $n>0$, are higher order corrections resolving the atomistic-continuum error. We will give precise definitions of all the $G_n$ in \S~\ref{sec:greensfunctions}. In addition we want to point out that the $G_n$ are practically computable via Fourier methods. A formula for that is also given in \S~\ref{sec:greensfunctions}.

Returning to the multipole expansion, if $p=1,2$ in Theorem \ref{thm:structurewithmoments}, then one can replace the lattice Green's function $\G$ with the continuum Green's function $G^{\rm C}$. However, for a higher order continuum description, higher order $G_n$ need to be used. If we also use Taylor expansions to get actual derivatives we get a pure continuum expansion.

\begin{theorem} \label{thm:structurewithmomentscont}
Assume $\lvert H [u]  \rvert \lesssim \abs{\ell}^{-d-p} \log^\alpha \abs{\ell}$ with $p, \alpha \in \N_0$. Then there are $a^{(i,n,k)} \in (\R^{d})^{\odot i}$ such that 
\[u = \sum_{k=1}^N \sum_{n=0}^{\big\lfloor \frac{p-1}{2} \big\rfloor} \sum_{i=0}^{p-1-2n} a^{(i,n,k)} : \nabla^i (G_{n})_{\cdot k} + \tilde{w} \]
and the remainder decays as 
\begin{equation} \label{eq:structurewithmomentscontremainder}
 \lvert D^j \tilde{w} \rvert \lesssim \abs{\ell}^{2-d-p-j}  \log^{\alpha +1} \abs{\ell}
\end{equation}
for $j=1,2$.

If $j \geq 3$, then equation \eqref{eq:structurewithmomentscontremainder} is still true under the additional assumption that $\lvert D^m H [u] \rvert \lesssim \abs{\ell}^{-d-p-m} \log^\alpha \abs{\ell}$ for all $m=1,...,j-2$.
\end{theorem}

\subsection{The Full Far-Field Expansion}
\label{sec:contdev}
Theorem \ref{thm:structurewithmoments} lays out a path on how to construct good far-field approximations for a solution $\bar{u}$ of the atomistic equations $\delta \mathcal{E} (\bar{u})=0$. Instead of looking at the solutions directly it suffices to construct an approximate solution of the equations $\hat{u}$ such that the remainder $r$ defined by $\bar{u}=: \hat{u} + r$ has small linearized forces $H[r]$ in the far field. For this approach to be useful, it is desirable that $\hat{u}$ is both easy to understand analytically and practically computable. 

Our goal is to construct smooth continuum approximations through the addition of successive corrector terms $u_i^{\rm C}$ up to the desired order. We write $\bar{u}=\hat{u}_p+r_p= u_0^{\rm C} + u_1^{\rm C} + \dots + u_{p}^{\rm C} + r_p$ and aim to achieve $\lvert H[r_p] \rvert \lesssim \abs{\ell}^{-d-p-1}$, so that with Theorem \ref{thm:structurewithmoments} we have
\[\bar{u}= \sum_{i=0}^p u_i^{\rm C} + \sum_{i=1}^{p} \sum_{k=1}^N b^{(i,k)} : D_{\rm \mathcal{S}}^i \mathcal{G}_k + w_p\]
with a remainder $w_p$ satisfying $\lvert w_p \rvert \lesssim \abs{\ell}^{1-d-j-p}$; that is, the remainder is highly localised around the defect core.

The precise statements for both point defects and screw dislocations are given in \S~\ref{sec:applications}. The full rigorous construction of the $u_i^{\rm C}$ is given in the proofs in \S~\ref{sec:prooffarfieldexpansion}. We do however want to formally outline this construction here.

The first step is a Taylor expansion of the energy around the lattice or, more precisely, the potential $V$ around zero, giving
\[0 = \delta \mathcal{E}(\bar{u}) = \sum_{k=1}^{\tilde{K}} \frac{1}{k!} \delta^{k+1} \mathcal{E}(0)[\bar{u}]^{k} + h.o.t.\]
Separating out the linear terms and inserting the ansatz $\bar{u}= \hat{u}_p + r_p = u_0 + u_1 + \dots + u_{p} + r_p$ gives
\[ \sum_{i=0}^{p} H[u_i] + H[r_{p}] = -\sum_{k=2}^{\tilde{K}} \frac{1}{k!} \delta^{k+1} \mathcal{E}(0)[\hat{u}_p]^{k} + h.o.t.\]
The next step then is to Taylor expand the discrete differences in $H[u_i]$ and $\delta^{k+1} \mathcal{E}(0)[\hat{u}_p]^{k}$ leading to continuum differential operators. In particular, the leading order term for $H[u_i]$ is $c_{\rm vol} H^{\rm C}[u_i]$ followed by higher order differential operators.

Formally, $u_i$ is of the order $\abs{\ell}^{2-d-i}$ with each derivative or discrete difference adding one order of decay. This allows us to group all the resulting terms into  $\mathcal{S}_i$, based on their order of decay. We thus obtain
\begin{equation} \label{eq:Hrpequation}
H[r_{p-1}] = \sum_{i=0}^p c_{\rm vol} \Big(\mathcal{S}_i(u_0, \dots, u_{i-1}) - H^{\rm C}[u_i]\Big) + h.o.t.
\end{equation}
We can therefore use the PDE
\begin{equation} \label{eq:correctorequation}
H^{\rm C}[u_i^{\rm C}] = \mathcal{S}_i(u_0, \dots, u_{i-1})
\end{equation}
to define $u_i^{\rm C}$. When we look at all the details of this construction in \S~\ref{sec:prooffarfieldexpansion}, we will see that the $u_i$ as used on the right hand side of \eqref{eq:correctorequation} has to include the multipole terms of that order, which we can write as
\[u_i = u_i^{\rm C} + u_i^{\rm MP}.\]
In particular, it is then important to know that with the help of Theorem \ref{thm:structurewithmomentscont} we can use the their smooth continuum variant instead of the discrete one.

The $\mathcal{S}_i$ in \eqref{eq:correctorequation} are in general non-linear and higher order differential operators. Crucially though, $\mathcal{S}_i$ only depends on the terms $u_0, \dots, u_{i-1}$ and not $u_i$ itself. That means that the equation defining $u_i$ is always  the same second order, elliptic continuum linear elasticity equation $H^{\rm C}[u_i^{\rm C}] = f_i$ for some residual forces $f_i$.

With $u_i^{\rm C}$ defined this way, most of the terms on the right hand side of \eqref{eq:Hrpequation} cancel out and a precise estimate of the higher order errors gives the desired estimate for $H[r_{p-1}]$.

In \S~\ref{sec:prooffarfieldexpansion}, we give a precise definition of the $\mathcal{S}_i$ in the corrector equation \eqref{eq:correctorequation}. The grouping of the terms into different $\mathcal{S}_i$ depends on the dimension $d$. As an example, for $d=2$, the first three $S_i$ are given by
\begin{align}
\mathcal{S}_0&=0,\\
\mathcal{S}_1(u_0^{\rm C}) &=\frac{1}{2}\divo \big(\nabla^3 W (0)[\nabla u_0^{\rm C}]^2\big)\\
\mathcal{S}_2(u_0^{\rm C}, u_1^{\rm C}, u_1^{\rm CMP}) &= \divo \big(\nabla^3 W (0)[\nabla u_0^{\rm C},\nabla u_1^{\rm C}+u_1^{\rm CMP} ]\big)\\
&\quad +\frac{1}{6}\divo \big(\nabla^4 W (0)[\nabla u_0^{\rm C}]^3\big)\nonumber \\
&\quad - H_{\rm SG}[u_0^{\rm C}] \nonumber
\end{align}  

Recall from \eqref{eq:defn_W_cb} that $W(A) = \frac{1}{c_{\rm vol}} V((A\rho)_{\rho \in \Rc})$ is the Cauchy-Born energy density for $A\in \R^{N \times d}$ and $H_{\rm SG}$ is a linear differential operator describing a strain-gradient term in linear elasticity. It is defined by
\begin{align*} 
  H_{\rm SG}[u] := \frac{1}{12 c_{\rm vol}} \sum_{\sigma, \rho \in \Rc} \nabla^2 V(0)_{\sigma \rho} \big( 3\nabla^4 u[\sigma,\sigma,\rho,\rho] -2\nabla^4 u[\sigma,\rho,\rho,\rho]&  \\
        -2\nabla^4 u[\sigma,\sigma,\sigma,\rho] \,& \big) .
\end{align*}

If on the other hand $d=3$, then
\begin{align}
\mathcal{S}_0&=0,\\
\mathcal{S}_1&=0,\\
\mathcal{S}_2&=\frac{1}{2}\divo \nabla^3 W(0)[\nabla u_0^{\rm C}]^2 - H_{\rm SG}[u_0^{\rm C}].
\end{align} 

\section{Far Field Expansion for Crystalline Defects}
\label{sec:applications}
We now demonstrate how our general structural results can be directly applied to obtain precise characterisations of the discrete elastic far-fields surrounding crystalline defects. Our results apply directly to points defects and screw dislocations; edge and mixed mode dislocations require additional ideas due to their more non-trivial lattice topology and will therefore not be discussed here. In addition, we briefly outline how these characterisations give rise to novel algorithms for simulating such defects.

\subsection{Point defects}
\label{sec:results:point}
We consider point defects first. We briefly review the setting of \cite{EOS2016} to motivate the formulation of our main result in this context. First, we assume that the point defect has a reference configuration $\Ladef \subset \R^d, d \geq 2$ which is locally finite and homogeneous outside some defect radius $\Rdef$, meaning $\Ladef \setminus B_{\Rdef} = \Lambda\setminus B_{\Rdef}$.

Let $\Rc_\ell$ denote a finite interaction range for each site $x \in \Ladef$ and assume that there is a family of site energies $V_\ell \in C^K(\R^{d \Rc_\ell}), \ell \in \Ladef$. Moreover, assume $N=d$ and that there is a homogeneous interaction range $\Rc$ and site energy $V \in C^K(\R^{d \Rc})$ for all sites of $\La$, and that $\Rc_\ell = \Rc, V_\ell =  V$ for all $\ell \in \Ladef \setminus B_{\Rdef}$. The potential energy under a displacements $u : \Ladef \to \R^d$ and $u : \La \to \R^d$ are then, respectively, given by
\begin{align*}
   \mathcal{E}^{\rm def}(u) &:= \sum_{\ell \in \Ladef} \big[V_\ell (D_{\mathcal{R}_\ell} u(\ell)) - V_\ell(0)\big], \\
   \mathcal{E}(u) &:= \sum_{\ell \in \La} \big[ V (D_{\mathcal{R}} u(\ell)) - V(0) \big].
\end{align*}
With these definitions $\mathcal{E}^{\rm def}$ is then well defined on $\HHc$ and have a unique continuous extension to $\HH$. The same is true with $\mathcal{E}^{\rm def}$ for analogously defined $\HHc(\Ladef)$ and $\HH(\Ladef)$. Furthermore with $V \in C^K$ we also find $\mathcal{E}, \mathcal{E}^{\rm def} \in C^K$, all of which is shown in \cite[Lemma 1]{EOS2016}.

Allowing $\Ladef \neq \La$ in $B_{\Rdef}$ admits defects such as vacancies and interstitials, while allowing inhomogeneity of $V_{\ell}$ admits impurities and foreign interstitials.

A {\em point defect} can be thought of as a finite-energy equilibrium of $\mathcal{E}$, that is, equilibrium displacements $\bar{u}^{\rm def} \in \HH(\Ladef)$ such that
\begin{equation}
  \label{eq:equil_point}
   \delta \mathcal{E}^{\rm def}(\bar{u}^{\rm def})[v] = 0 \qquad \forall v \in \HHc(\Ladef).
\end{equation}
A notationally convenient approach is to simply project $\bar{u}^{\rm def}$ to the homogeneous lattice. That is, we define $\bar{u} : \La \to \R^d$ by 
\[
   \bar{u}(\ell) := \begin{cases}
      \bar{u}^{\rm def}(\ell), & \ell \in \La \cap \Ladef, \\
      0, & \ell \in \La \setminus \Ladef.
   \end{cases}
\]
This is of course only one of many possible projections, which we have made only for the sake of notational convenience. Our subsequent results are essentially independent of how this projection is performed. Most importantly, because we have $\bar{u} = \bar{u}^{\rm def}$ outside the defect core we obtain that
\[
   \delta \mathcal{E}(\bar{u})(\ell) = \delta \mathcal{E}(\bar{u})[\delta_\ell] = 0,
\]
for $\lvert \ell \rvert$ large enough. Here $\delta_\ell(\ell') := \delta_{\ell\ell'}$. With a small amount of additional work one can in fact show that
\[
   \delta \mathcal{E}(\bar{u})[v] = (g, Dv)_{\ell^2} \qquad \forall v \in \HH(\La),
\]
where $g : \La \to \R^{d \times \Rc}$ with ${\rm supp}(g) \subset B_{\Rdef}$. This motivates the setup for our next result.

\begin{theorem} \label{thm:pointdef}
   Choose $p \geq 0, J \geq 0$ and suppose that $V \in C^{K}(\R^{d \times  \Rc})$, such that 
   $K \geq J+2+ \max \{0,\lfloor \frac{p-1}{d} \rfloor\}$. Let $g : \La \to \R^{d \times \Rc}$ with compact support, and let
    $\bar{u} \in \HH(\La)$ such that
   \[
      \delta \mathcal{E}(\bar{u})[v] = ( g, D v )_{\ell^2}
      \quad \forall v \in \HHc(\La).
   \]
   Then, there exist $u_i^{\rm C} \in C^\infty$ and coefficients $b^{(i,k)}$ such that 
\[
\bar{u} = \sum_{i=d+1}^p u_i^C + \sum_{i=1}^{p} \sum_{k=1}^N b^{(i,k)} : D_{\rm \mathcal{S}}^i \mathcal{G}_k + r_{p+1},
\]   
and such that the $u_i^C$ satisfy the PDEs in equation \eqref{eq:uCiPDE} with $u_i^C=0$ for $0 \leq i \leq d$. Furthermore, the remainder $r_{p+1}$ satisfies the estimate
   \[\lvert D^j r_{p+1} \rvert \lesssim \abs{\ell}^{1-d-j-p}\log^{p+1} \abs{\ell}\]
for $j= 1, \dots, J$.   
\end{theorem}

\begin{remark} \label{rem:pointdef}
In particular, up to order $p=d$, we have the pure multipole expansion
\begin{align} \label{eq:puremultipole}
\bar{u} = \sum_{i=1}^{d} \sum_{k=1}^N b^{(i,k)} : D_{\rm \mathcal{S}}^i \mathcal{G}_k + r_{d+1},
\ \text{where} \ 
\lvert D^j r_{d+1} \rvert \lesssim \abs{\ell}^{1-2d-j} \log^{d+1} \abs{\ell}.
\end{align}
Effects from the nonlinearity and higher order derivatives are only noticeable in terms beyond that.
\end{remark}

\begin{remark} \label{rem:pointdef2}
In the point defect case it is likely possible to reduce the number of logarithms in the estimate somewhat for all orders. We do not want to explore this in detail but want to point out the log-free estimate for low orders which directly follows from \eqref{eq:puremultipole} and estimates on the lattice Green's function based on Theorem \ref{thm:greensfunctionexpansion}. To be precise, for $p<d$ we have
\begin{align*}
\bar{u} = \sum_{i=1}^{p} \sum_{k=1}^N b^{(i,k)} : D_{\rm \mathcal{S}}^i \mathcal{G}_k + r_{p+1},
\ \text{where} \
\lvert D^j r_{p+1} \rvert \lesssim \abs{\ell}^{1-d-p-j}.
\end{align*}
\end{remark}

\subsection{Screw dislocations} \label{sec:results:screw}

Now let us consider screw dislocations. Again, our modelling follows the setup in \cite{EOS2016} and \cite{BBO19}.
We consider a straight screw dislocation with periodic behaviour along the dislocation line so that we can project to the lattice to a two-dimensional lattice on the normal plane to describe the behaviour. Hence we have $d=2$ and $N=3$ meaning $u \colon \La \subset \R^2 \to \R^3$, though $N$ is left arbitrary in the following to include for example the case in \cite{BBO19} where $N=1$.

Again we have a finite interaction range $\Rc$ and a site energy $V \in C^K(\R^{d \Rc})$. The potential energy is then formally given by
\begin{align*}
   \mathcal{E}(u) &= \sum_{\ell \in \La} V (D_{\mathcal{R}} u(\ell)).
\end{align*}
However that sum will usually not converge so we follow \cite{EOS2016} and consider the energy differences instead,
\begin{equation} \label{eq:energydislocation}
   \mathcal{E}(u) := \sum_{\ell \in \La} \Big(V (D_{\mathcal{R}} u(\ell))-V (D_{\mathcal{R}} u_{\rm CLE}(\ell))\Big),
\end{equation}
where $u_{\rm CLE}$ is the continuum linear elasticity solution.

More precisely, $u_{\rm CLE}$ solves 
\begin{align} 
   -\divo \mathbb{C}[\nabla u] = 0&,  \qquad \text{in } \R^2 \setminus \Gamma, \label{eq:clescrewPDE}\\
   u(x+) - u(x-) = - b&,  \qquad \text{for } x \in \Gamma \setminus \{\hat{x}\}, \label{eq:clescrewjump}\\
   \nabla_{e_2} u(x+) - \nabla_{e_2} u(x-) = 0&, \qquad \text{for } x \in \Gamma \setminus \{\hat{x}\} \label{eq:clescrewgradientatcut},\\
         \lvert \nabla u \rvert \to 0&,  \qquad \text{for } \lvert x \rvert \to \infty \label{eq:clescrewgradientatinfinity}\\
         -\int_{\partial B_1(\hat{x})} \mathbb{C}[\nabla u] \nu \, d\sigma =0&. \label{eq:clescrewnetforce}
\end{align}
where $b \in \R^3$, $b \parallel e_3$, is the Burgers vector of the screw dislocation, $\hat{x} \in \R^2$ is the reference position of the dislocation core
and $\Gamma := \{ x \in \R^2 : x_2 = \hat{x}_2, x_1 \geq \hat{x}_1 \}$
a branch-cut chosen such that $\Gamma \cap \Lambda = \emptyset$. We want to point out that the precise positioning of $\hat{x}$ is not crucial and does not have physical meaning as the difference between two shifted solutions is in the energy space $\HH$.

Equation \eqref{eq:clescrewnetforce} was missed in \cite{EOS2016} but is in fact crucial for the results there to be true. It encodes the assumption that the system has zero net force and thus avoids spurious solutions of the type $g(x)= u_{\rm CLE}(x) + G^C(x-\hat{x})$. However, the standard construction of a solution $u_{\rm CLE}$, which can be found, e.g., in the book by Hirth and Lothe \cite{hirth-lothe}, already takes it into account. Therefore, the results of \cite{EOS2016}  and later works that build on it remain correct provided such a solution $u_{\rm CLE}$ is employed.

The following observation links it to the atomistic setting.
\begin{proposition} \label{prop:uCLE}
Let $u \in C^3(\R^2 \backslash \Gamma), \nabla u \in C^2(\R^2 \setminus \{\hat{x}\})$ solve \eqref{eq:clescrewPDE}, \eqref{eq:clescrewjump}, and \eqref{eq:clescrewgradientatcut} with $\lvert \nabla^j u \rvert \lesssim \lvert \ell-\hat{x} \rvert^{-j}$ for $j=1,2,3$. Then

\[-\int_{\partial B_1(\hat{x})} \mathbb{C}[\nabla u] \nu \, d\sigma = \sum_{\ell \in \Lambda} \delta \mathcal{E}(u)(\ell) = \sum_{\ell \in \Lambda} H[u](\ell). \]
Therefore \eqref{eq:clescrewnetforce} is equivalent to either of these sums vanishing.
\end{proposition}

Indeed, the property $\sum \delta \mathcal{E}(u)(\ell) =0$ is heavily used in \cite{EOS2016} and we will use their results here. 


With the definition \eqref{eq:energydislocation} the energy $\mathcal{E}$ is then well defined on $u_{\rm CLE} + \HHc$ and has a unique continuous extension to $u_{\rm CLE} + \HH$. And with $V \in C^K$ we also find $\mathcal{E} \in C^K$, see \cite[Lemma 3]{EOS2016}.

\begin{theorem} \label{thm:screw}
   Suppose that $V \in C^{K}(\R^{N \times\Rc})$, $K \geq J+2+ p$ with $p \geq 0$ and $J \geq 2$. Let $\bar{u} \in \HH(\La)$ solve
   \begin{equation} \label{eq:equil_screw}
      \delta \mathcal{E}(\bar{u})[v] = 0
      \quad \forall v \in \HHc(\La).
   \end{equation}
   Then, there exist $u_i^{\rm C} \in C^\infty$ and coefficients $b^{(i,k)}$ such that 
\[
\bar{u} = \sum_{i=0}^p u_i^C + \sum_{i=1}^{p} \sum_{k=1}^N b^{(i,k)} : D_{\rm \mathcal{S}}^i \mathcal{G}_k + r_{p+1},
\]   
and such that the $u_i^C$ satisfy the PDEs \eqref{eq:uCiPDE} and $u_0^{\rm C}=u_{\rm CLE}$. Furthermore, the remainder $r_{p+1}$ satisfies the estimate
   \[\lvert D^j r_{p+1} \rvert \lesssim \abs{\ell}^{-1-j-p}\log^{p+1} \abs{\ell}\]
for $j= 1, \dots, J$.   
\end{theorem}

\begin{remark}
Contrary to the point defect none of these terms are expected to vanish in general, except for a few special cases which are explored in \cite{BBO19}. In particular, the regularity assumption cannot be weakened as in the point defect case. Indeed, our general theory  without looking at any special cases requires $K \geq J+2+ \lfloor \frac{p}{d-1} \rfloor$. As far as our proof goes the number of logarithms is optimal for $d=2$, though probably not for higher dimension. We also expect this to be generic for the theorem itself as indeed the $u_i^C$ will (in general) contain higher and higher logarithmic terms. However, in special cases these logarithmic terms in the $u_i^C$ do not necessarily always appear, as explored in \cite{BBO19}.
\end{remark}

\subsection{Accelerated Convergence of Cell Problems}
\label{sec:cell_problems}
\def\calW{\mathcal{W}}
An immediate application of the defect expansions of Theorems~\ref{thm:pointdef} and~\ref{thm:screw} is that they
suggest a novel family of numerical schemes that exploit these expansions to accelerate the simulation of crystalline
defects. Here, we will only sketch one such scheme, but leave a more
detailed analysis for future work.

Consider the equilibration of a point defect or a screw dislocation
near the origin as in  Theorems~\ref{thm:pointdef} and~\ref{thm:screw}.
We define a family of restricted displacement spaces
\begin{align}
   \calW_R &:= \big\{  v : \Lambda \to \R^N \,|\,
                    v(\ell) = 0 \text{ for $|\ell| > R$} \big\}, 
                  \\
    \mathcal{U}_R &:= \big\{ u = u_0^C + v \,|\, v \in \calW_R \big\},
\end{align}
where atoms are clamped in their reference configurations outside a ball
with radius $R$.  Then we can approximate \eqref{eq:equil_point}, \eqref{eq:equil_screw} by the Galerkin projection
\begin{equation}
   \label{eq:cellp:galerkin}
   \delta \mathcal{E}(\bar{u}_R)[v] = 0 \qquad \forall v \in \calW_R,
\end{equation}
where $\bar{u}_R \in \mathcal{U}_R$. 

Under suitable stability conditions it is then shown in \cite{EOS2016} that
\begin{equation} \label{eq:slow_convergence_cell}
   \| D\bar{u}_R - D\bar{u} \|_{\ell^2} \leq C R^{-d/2} \log^p R,
\end{equation}
for $R$ sufficiently large, where $p \in \{0, 1\}$. This convergence is an almost immediate corollary of the
decay estimate $|Dr_1(\ell)| \lesssim \abs{\ell}^{-d} \log^p \abs{\ell}$. (For energy minima,
\cite{EOS2016} can be applied directly while for saddle points the analysis of
\cite{BDO20} can be readily adapted.) Our aim now is to accelerate this
relatively slow convergence by providing an improved far-field boundary
condition. 

The overarching principle is to
\begin{enumerate} 
  \item replace the naive far-field predictor 
\[
   \hat{u}_0 := \begin{cases}
      0, & \text{point defects}, \\ 
      u_{\rm CLE}, & \text{dislocations}
   \end{cases}
\]  
with the higher-order predictor
\[
  \hat{u}_p := \sum_{i = 0}^p u_i^{\rm C}
\]
   \item and to enlarge the admissible corrector space with the multipole moments 
  \begin{align*}
    \mathcal{U}_R^{(p)} := \big\{
        v : \Lambda \to \mathbb{R}^N \,|\, & \,\,
        v = \sum_{i = 1}^p \sum_{k = 1}^N b^{(i,k)} : D^i_{\mathcal{S}} \mathcal{G}_k
              + w, \\ 
          & \text{for free coefficients $b^{(i,k)}$ and ${\rm supp}~ w \subset \Lambda \cap B_R$ }
    \big\}
  \end{align*}
  That is, the corrector displacement is now parametrised by its values in 
  the computational domain $B_R \cap \Lambda$ and by the coefficients 
  of the multipole terms.
\end{enumerate}

\noindent We can then consider the pure Galerkin approximation scheme: find 
$\bar{u}_{p, R} \in \mathcal{U}^{(p)}_R$ such that 
\begin{equation} \label{eq:galerkin}
    \delta \mathcal{E}(\bar{u}_{p, R}) [v_R] = 0 \qquad \forall v_R \in \calW_R. 
\end{equation}

The arguments of \cite{EOS2016} leading to \eqref{eq:slow_convergence_cell} are generic Galerkin approximation arguments, leveraging the strong stability condition. They can be followed {\em verbatim} up to the intermediate result (C\'{e}a's Lemma)
\[
  \big\| D\bar{u} - D\hat{u}_p - D\bar{v}_R \big\|_{\ell^2}
  \leq C \inf_{v_R \in \mathcal{U}_R^{(p)}}
  \big\| D\bar{u} - D\hat{u}_p - Dv_R \big\|_{\ell^2}.
\]
The existence of $\bar{v}_R$ is implicitly guaranteed through an application of the inverse function theorem, due to the fact that the right-hand side in this estimate approaches zero as $R \to \infty$.
To estimate the right-hand side we can insert the {\em exact} tensors $b^{i,k}$ from the solution representation of Theorems~\ref{thm:pointdef}, \ref{thm:screw} into $v_R$, in order to obtain 
$D\bar{u} - D\hat{u}_p - Dv_R = Dr_{p+1} - Dw_R$, where $r_{p+1}$ is the core remainder term,  
and hence 
\[
  \inf_{v_R \in \mathcal{U}_R}
  \big\| D\bar{u} - D\bar{u}_{p, R} \big\|_{\ell^2} \leq 
  \inf_{w_R \in \mathcal{U}_R}
  \big\| Dr_{p+1} - Dw_R \big\|_{\ell^2},
\]
We can now define $w_R$ to be a suitable truncation of $r_{p+1}$ to the computational domain $B_R$. The details are given in \cite[Thm. 2]{EOS2016} and immediately yield the following result. 

\begin{theorem}  \label{th:galerkin}
  Suppose that $\bar{u}$ is a strongly stable solution of \eqref{eq:equil_point} or \eqref{eq:equil_screw}; that is, there exists a stability constant $c_0 > 0$ such that 
  \[
      \delta^2 \mathcal{E}(\bar{u})[v,v] 
      \geq c_0 \| Dv \|^2, \qquad \forall v \in \HH(\Lambda),
  \]
  then, for $R$ sufficiently large, there also exists a solution $\bar{u}_R \in \mathcal{U}_R^{(p)}$ to the Galerkin scheme \eqref{eq:galerkin} such that 
  \[
     \big\| D\bar{u} - D\bar{u}_{p,R} \big\|_{\ell^2}
     \leq C_p
     R^{- d/2 - p} \log^{p+1} R.
  \]
\end{theorem}

\begin{remark}
  The scheme \eqref{eq:galerkin} cannot be implemented as is since the energy difference functional cannot be evaluated for a displacement with infinite range. However, this highly idealised scheme is of immense theoretical value in that it highlight what could potentially be achieved if this challenge can  be overcome. Any practical scheme will necessarily have to engage in the approximate evaluation of the multipole tensors $b^{(i)}$, for which there are several promising possibilities that we will explore in separate works. 
  
  A second challenge for practical implementations is the fast and accurate evaluation of the higher-order far-field predictor $\hat{u}_p$. All of these approximations require suitable controlled approximation to $\mathcal{E}$, somewhat analogous to quadrature rules or other kinds of variational crimes in the classical numerical analysis context. 
\end{remark}

\section{Conclusions and Outlook}
\label{sec:conclusions}
The main result of the present paper is the fact that the elastic field surrounding a defect in a crystalline solid may be represented to within arbitrary accuracy with three ``low-dimensional'' ingredients: 
\begin{enumerate}
  \item a series of continuum fields specified through PDEs; 
  \item a series of multipole moments; and
  \item a highly localised discrete core correction.
\end{enumerate}
More specifically, we have shown that by increasing the accuracy of components (1) and (2), the core correction (3) becomes increasingly local. While there is a certain amount of interaction between the components (1) and (2), there is no coupled problem that needs to be solved at any point. Indeed, both series are obtained sequentially order by order and the PDE defining the term of a given order in (1) only depends on lower-order terms of the multipole expansion, but not on the multipole term of the same order.

Our presentation here is restricted to simple lattices and a limited class of defects. Generalisations do require additional technical difficulties to be overcome, but there appears to be no fundamental limitation to extend the method and the results to multi-lattices and a range of other defects in some form.
 To conclude, we briefly discuss some of these possibilities and limitations.

\begin{itemize} 
  \item {\it Edge and mixed dislocations:} Edge and mixed dislocations are technically more challenging as they create a mismatch that affects the two-dimensional reference lattice. To leading order it suffices to correct the CLE solution with an ad-hoc transformation $u_0 = u_{\rm CLE} \circ \xi^{-1}$, see \cite{EOS2016}, though the analysis also becomes a bit more technical still. For higher orders however more care needs to be taken, not just in the choice of $\xi$ but also in the effect such transformations have on the PDEs. Furthermore, many arguments have additional technical complications due to the need of slip operators to describe the elastic strain.
  %
  \item {\it Cracks: } The full extension of our results to crack geometries appears to be considerably more challenging, as the homogeneous lattice is no longer a particularly good global reference. Thus already the discussion of Green's function is significantly more involved~\cite{BHO19}. Formally, we still expect our overall strategy to apply and it is interesting to note that due to different orders of decay the first higher order corrector is already needed to even define the model in the first place rather than ``just'' improve on it. 
  \item {\it Energy differences for defect transitions: } The precise characterisation of the far-field strain in terms of defect continuum fields and a multipole expansion suggests that some level of cancellation in energy differences, e.g. between a saddle point and energy minimum as observed in \cite{BDO20}, could be precisely tracked and characterised. Moreover, such results may then also explain improved convergence rates of numerical schemes for energy differences that are often seen in practice.
  \item {\it Convergence of numerical schemes: } A consequence of our analysis, with direct practical value is the construction of improved approximate cell problems that leverage the explicit low-dimensional structure of defect fields that we identified. We have given a hint at how this might be achieved in Theorem~\ref{th:galerkin} but much additional work is needed to formulate practical schemes along these lines. 

	The same line of work can also lead to robust new numerical schemes and analysis of existing schemes for the defect dipole tensor specifically (also called the elastic dipole tensor). Such schemes are of important and ongoing interest in defect physics (e.g., \cite{nazarov16}, \cite{DM18}). In particular, our approach to these terms developed here naturally includes the anisotropic case as well as extensions to higher multipole tensors.
\item {\it Higher-order dislocation dynamics models: } A further consequence is that we hope that the expansion of the far-field strain we have obtained here allows us to go beyond traditional dislocation dynamics approaches, which rely upon the leading-order CLE description of these defects. 
  By using the structure of our expansion to studying the effect of applied stress fields on defect cores, we can provide more detailed atomistic input into such models. This suggests a route to better connect dislocation dynamics and atomistic approaches, bridging the scale and language gaps between these two simulation methodologies.

  \item {\it Dynamics and statistical mechanics: } Statistical mechanics models, such as free energies or transition rates could in principle benefit from an analysis within our new framework. For example, in the harmonic approximation, the analysis of \cite{BDO20} could be taken as a starting point. It is far less clear whether more nonlinear models could also benefit, and it appears certain that finding similar coarse-grained descriptions of full dynamics of a crystalline far-field would require very different ideas.
\end{itemize}

\section{Proofs - Decay Estimates}
\label{sec:proofs-decay}

In this section we want to prove Theorem \ref{thm:structure} and Theorem \ref{thm:structurewithmoments}. But first, let us cite the following lemma.

\begin{lemma}[Conversion to divergence form] \label{divlemma}
Let $\alpha \in \R$, $q > d$, and $f\,:\,\La \to \R^N$ such that $|f(\ell)| \leq C_f \abs{\ell}^{-q} \log^{\alpha}  \abs{\ell}$ for all $\ell \in \La$ and $\sum_{\ell\in\La}f(\ell) = 0$. Then there exists $g\,:\,\La \to \R^{N} \otimes \R^{\mathcal{S}}$ and a constant $C$ independent of $f$ and $\ell$ such that $f = -{\rm Div}_{\mathcal{S}} g$ and $ \lvert g(\ell) \rvert \leq CC_f\abs{\ell}^{-q+1} \log^{\alpha}  \abs{\ell}$ for all $\ell\in\La$.
\end{lemma}
\begin{proof}
For $\alpha =0$ this is \cite[Corollary 1]{EOS2016}. However the addition of logarithmic terms is trivial. Indeed, one can construct $g$ in the exact same way and can easily carry the logarithmic term through by including them in the weighted norms used in the proof.
\end{proof}

The decay estimate for the lowest order found in \cite{EOS2016} is indeed based on Lemma \ref{divlemma} which then allows for a partial summation in the Green's function representation sum of the remainder given in equation \eqref{greensfunction}. The key idea for the higher order decay estimates is that Lemma \ref{divlemma} can actually be extended to higher orders based on vanishing higher order moments
\[
\mathcal{I}_j = \sum_{\ell \in \La} f(\ell) \otimes \ell^{\otimes j},
\]
as long as one only tries to write symmetric parts in divergence form; see Proposition \ref{divlemmahigherorder} below. We will then use this higher order divergence form with more precise higher order partial summation in specific parts of the lattice in the Green's function representation sum of the remainder given in equation \eqref{greensfunction} to arrive at the new decay estimates. 
We begin by establishing two further auxiliary results.

In the following, we will use $\boldsymbol{\rho} = (\rho_1, ..., \rho_j) \in \mathcal{R}^j\subset(\R^d)^j$, $j \geq 1$, and analogously for $\boldsymbol{\sigma}$. Also recall that $D_{\boldsymbol{\rho}} = D_{\rho_1} ... D_{\rho_j}$. 

\begin{lemma} \label{symmetrictensors}
Given a linearly independent set of vectors $\mathcal S \subset \R^d$ (which must necessarily have $\#\mathcal S=k\leq d$), the set of tensors $\{ \boldsymbol{\sigma}^\odot \colon \boldsymbol{\sigma} \in \mathcal{S}^k\}$ is also linearly independent in $(\R^d)^{\otimes k}$. Furthermore, $\boldsymbol{\sigma}^\odot = \boldsymbol{\rho}^\odot$ with $\boldsymbol{\sigma}, \boldsymbol{\rho} \in \mathcal{S}^k$, if and only if $\boldsymbol{\rho}  = \pi (\boldsymbol{\sigma})$ for some permutation $\pi \in S_k$.
\end{lemma}
\begin{proof}
  Although the proof is straightforward, and likely well-known, we present it for the sake of convenience: Define a scalar product $(\cdot, \cdot)_{\mathcal S}$ on $\R^d$ for which $\mathcal S$ forms part of an orthonormal basis. This induces a scalar product on the space of tensors by multi-linear extension of $(\sigma^\otimes, \rho^\otimes)_{\mathcal S} := \prod_{j=1}^k (\sigma^{(j)}, \rho^{(j)})_{\mathcal S}$.

  Consider the scalar product
  \begin{align*}
    (\boldsymbol{\sigma}^\odot, \boldsymbol{\rho}^\odot)_{\mathcal S}
    &=\frac{1}{(k!)^2}\sum_{\pi',\pi''\in S_k}\left( \pi'(\boldsymbol{\sigma})^\otimes,\pi''(\boldsymbol{\rho})^\otimes\right)_{\mathcal S}\\
    &=\frac{1}{(k!)^2}\sum_{\pi',\pi''\in S_k}\prod_{i=1}^k\left( \sigma^{(\pi'(i))},\rho^{(\pi''(i))}\right)_{\mathcal S}\\
    &=\frac{1}{k!}\sum_{\pi\in S_k}\prod_{i=1}^k\left( \sigma^{(i)},\rho^{(\pi(i))}\right)_{\mathcal S}.
  \end{align*}
  If $\boldsymbol{\rho}\neq \pi(\boldsymbol{\sigma})$ for all $\pi\in S_k$, then for each $\pi\in S_k$, there exists an index $i$ such that $\sigma^{(i)}\neq \rho^{(\pi(i))}$, and consequently $(\sigma^{(i)},\rho^{(\pi(i))})_{\mathcal S}=0$. This entails that each of the products summed in the expression above is zero. Since $\boldsymbol{\rho}^\odot$ and $\boldsymbol{\sigma}^\odot$ are both non-zero and orthogonal in this inner product, they cannot be equal.

  The same argument entails that if $\boldsymbol{\sigma}^\odot = \boldsymbol{\rho}^\odot$, there must exist $\pi\in S_k$ such that $\boldsymbol{\rho}= \pi(\boldsymbol{\sigma})$. On the other hand, if $\boldsymbol{\rho}= \pi(\boldsymbol{\sigma})$, we see that
  \[
    \boldsymbol{\rho}^\odot=\pi(\boldsymbol{\sigma})^\odot=\frac{1}{k!}\sum_{\pi'\in S_k}\sigma^{(\pi'\pi(i))}\frac{1}{k!}\sum_{\pi''\in S_k}\sigma^{(\pi''(i))} = \boldsymbol{\sigma}^\odot.
  \]
  We deduce that $\boldsymbol{\sigma}^\odot$ and $\boldsymbol{\rho}^\odot$ are either identical (if and only if $\boldsymbol{\sigma}$ is a permutation of $\boldsymbol{\rho}$) or mutually orthogonal in the $\mathcal{S}$-inner product, which implies the stated result.
\end{proof}

\begin{lemma} \label{th:Diffpellp_identity}
  Let $\boldsymbol{\sigma} \in \Lambda^p$, then 
  \[
      D_{\boldsymbol{\sigma}} \ell^{\otimes p} = 
      p! \boldsymbol{\sigma}^\odot.
  \]
\end{lemma}
\begin{proof}
  This identity follows from two observations: First,
  \[D_\rho \ell^{\otimes j} = j \ell^{\odot (j-1)} \odot \rho + r(\ell), \]
  for a polynomial $r$ in $\ell$ of degree at most $j-2$. Secondly,
  \[D_{\boldsymbol{\sigma}}^{j-1} r=0 \]
  for any such polynomial. A simple induction then shows the result.
\end{proof}

We can now turn towards to a crucial result converting a discrete force field into higher order divergence form.

We will slightly abuse notation in the following and let the symmetric part $\sym A$ of a tensor $A \in \R^{N} \otimes (\R^\mathcal{S})^{\otimes k}$ denote the symmetrical part in the later indices only, namely $\sym (A) := \frac{1}{k!}\sum_{\pi \in S_k} A_{\cdot, \pi(\boldsymbol{\sigma})}$. For higher order tensor fields we will follow the usual convention from the continuum that $\Div$ always applies to the last component.

\begin{proposition}[Conversion to higher order divergence form] \label{divlemmahigherorder}
  Let $\mathcal{S} \subset \La$ be linearly independent with $\spano_\Z \mathcal{S}=\La$. Also let
  $p, q \in \N$, $q+1-p > d$, $\alpha \in \R$, and $f^{(0)}\,:\,\La \to \R^N$ such that $|f^{(0)}(\ell)| \leq C_f\abs{\ell}^{-q} \log^{\alpha} \abs{\ell}$ for all $\ell \in \La$. Also assume $p$ moments vanish, i.e., $\mathcal{I}_j=0$ for $0 \leq j \leq p-1$. Then there exist $f^{(k)}\,:\,\La \to \R^{N} \otimes (\R^\mathcal{S})^{\otimes k}$, $1 \leq k \leq p$ and a constant $C$ such that 
  \[
      \sym f^{(k-1)} = -\Div_{\mathcal{S}} f^{(k)}
    \] 
    and
$ \lvert f^{(k)}(\ell) \rvert \leq CC_f \abs{\ell}^{-q+k} \log^{\alpha} \abs{\ell}$ for all $\ell \in\La$ and all $1 \leq k \leq p$.
\end{proposition}
\begin{proof}
$p=1$ is covered by Lemma~\ref{divlemma}. By induction, assume the statement is true for a $p \in \N$ and we now have $q-p>d$ as well as $p+1$ vanishing moments. In particular, we have already constructed the desired $f^{(0)}, \dots, f^{(p)}$. Now $|\sym f^{(p)}(\ell)| \lesssim \abs{\ell}^{-(q-p)}$ and $q-p>d$.

We claim that
\begin{equation}
  \label{eq:divlemmahigherorder_Ip_identity}
  \mathcal{I}_p = \sum_{\ell} \sum_{\boldsymbol{\sigma} \in \mathcal{S}^j} f^{(j)}_{\cdot, \boldsymbol{\sigma}} \otimes D_{\boldsymbol{\sigma}}( \ell^{\otimes p}) 
\end{equation}
for all $j=0, \dots, p$. This is clear for $j=0$. By induction, if it is true for $j<p$, then
\begin{align*}
\mathcal{I}_p &= \sum_{\ell} \sum_{\boldsymbol{\sigma} \in \mathcal{S}^j} f^{(j)}_{\cdot, \boldsymbol{\sigma}} \otimes D_{\boldsymbol{\sigma}}( \ell^{\otimes p})\\
&= \sum_{\ell} \sum_{\boldsymbol{\sigma} \in \mathcal{S}^j} \sym f^{(j)}_{\cdot, \boldsymbol{\sigma}} \otimes D_{\boldsymbol{\sigma}}( \ell^{\otimes p})\\
&= -\sum_{\ell} \sum_{\boldsymbol{\sigma} \in \mathcal{S}^j} (\Div_\mathcal{S} f^{(j+1)})_{\cdot, \boldsymbol{\sigma}} \otimes D_{\boldsymbol{\sigma}}( \ell^{\otimes p})\\
&= \sum_{\ell} \sum_{\boldsymbol{\sigma} \in \mathcal{S}^j, \rho \in \mathcal{S}} D_{-\rho} f^{(j+1)}_{\cdot, \boldsymbol{\sigma},\rho} \otimes D_{\boldsymbol{\sigma}}( \ell^{\otimes p})\\
&= \sum_{\ell} \sum_{\boldsymbol{\sigma} \in \mathcal{S}^j, \rho \in \mathcal{S}} f^{(j+1)}_{\cdot, \boldsymbol{\sigma},\rho} \otimes D_\rho D_{\boldsymbol{\sigma}}( \ell^{\otimes p})\\
&= \sum_{\ell} \sum_{\boldsymbol{\sigma} \in \mathcal{S}^{j+1}} f^{(j+1)}_{\cdot, \boldsymbol{\sigma}} \otimes D_{\boldsymbol{\sigma}}( \ell^{\otimes p}).
\end{align*}
Note that the decay of $f^{(j+1)}$ is needed not just for the sums to exist but also for the partial summation to be true.  This establishes \eqref{eq:divlemmahigherorder_Ip_identity}. 

Applying Lemma~\ref{th:Diffpellp_identity}, with $j = p$, and using that $\mathcal{I}_p=0$ we obtain 
\[0= \sum_{\ell} \sum_{\boldsymbol{\sigma} \in \mathcal{S}^p} f^{(p)}_{\cdot, \boldsymbol{\sigma}} \otimes \boldsymbol{\sigma}^\odot=\sum_{\ell} \sum_{\boldsymbol{\sigma} \in \mathcal{S}^p} (\sym f^{(p)})_{\cdot, \boldsymbol{\sigma}} \otimes \sym \boldsymbol{\sigma}^\otimes.\]
According to Lemma \ref{symmetrictensors}, the set of the tensors $\sym \boldsymbol{\sigma}^\otimes$ is linearly independent. Additionally, $\boldsymbol{\sigma}^\odot = \boldsymbol{\rho}^\odot$ implies $\sym f^{(p)}_{\cdot, \boldsymbol{\sigma}} = \sym f^{(p)}_{\cdot, \boldsymbol{\rho}}$ as $\boldsymbol{\rho}  = \pi (\boldsymbol{\sigma})$ for some permutation $\pi$. Therefore, $\sum_\ell \sym f^{(p)}(\ell) =0$ and we can apply Lemma~\ref{divlemma} to find $f^{(p+1)}$ with the desired properties.
\end{proof}

To prepare for the Proof of Theorem \ref{thm:structure}, we fix some $u \in \HH$ and write $f^{(0)} := H[u]$. We extend the lattice Green's function approach developed in \cite{EOS2016} to estimate $u$. The Green's function satisfies
\[
 \sum_{z\in\La}f^{(0)}(z)\G_k(\ell-z) = u_k(\ell)
\]
for all $u \in \HHc$. As the right hand side is not invariant under adding a constant to $u$, this cannot directly translate to general $u \in \HH$, but the situation looks better for derivatives.

\begin{lemma} \label{th:lemma_for_thm_structure}
  Let $u \in \HH$ and assume that $|f^{(0)}(\ell)| = \lvert H[u]\rvert \lesssim |\ell|^{-\gamma}_0$ for some $\gamma>1$. 
  Then, for all $\boldsymbol{\rho} = (\rho_1, ..., \rho_j) \in \mathcal{R}^j$, $j \geq 1$, we have  
  \begin{equation} \label{greensfunction}
    \sum_{z\in\La}f^{(0)}(z)D_{\boldsymbol{\rho}}\G_k(\ell-z) = D_{\boldsymbol{\rho}} u_k(\ell).
   \end{equation}   
\end{lemma}
\begin{proof}
Due to the decay assumption on $f^{(0)}$ the sum converges absolutely. Furthermore, for $u \in \HHc$ the statement is clearly true. The right hand side is a well-defined, continuous, linear functional on $\HH$. The result is straightforward when  $d \geq 3$ or $j \geq 2$ as in this case the left hand side is also a continuous, linear functional because $|D_{\boldsymbol{\rho}}\G_k| \in \ell^2(\La)$ if and only if $4 < d+2j$.

  To include the case where $d=2$ and $j=1$ we have to be a bit more careful. Note that, for $u \in \HHc$
  \[ \sum_{z\in\La}f^{(0)}(z)D_{\boldsymbol{\rho}}\G_k(\ell-z) = - \sum_{z\in\La} \nabla^2 V(0)[Du(z), D D_{\boldsymbol{\rho}}\G_k(\ell-z)]. \]
  As the right hand side is now a well-defined, continuous, linear functional on $\HH$, we find for all $u \in \HH$, that 
  \begin{equation} \label{greensfunction2}
   - \sum_{z\in\La} \nabla^2 V(0)[Du(z), D D_{\boldsymbol{\rho}}\G_k(\ell-z)] = D_{\boldsymbol{\rho}} u_k(\ell).
  \end{equation}
  If we now have a $u \in \HH$ where additionally $\lvert f^{(0)} \rvert \lesssim \abs{\ell}^{-\gamma}$, then the sum $\sum_{z\in\La}f^{(0)}(z)D_{\boldsymbol{\rho}}\G_k(\ell-z)$ converges (absolutely). Consider a smooth cutoff function $\eta_R$, such that $\eta_R \colon \R^d \to [0,1]$ satisfies $\eta_R(z)=1$ for $\lvert z \rvert \leq R$ and $\eta_R(z)=0$ for $\lvert z \rvert \geq 2R$, as well as $\lvert \nabla^k \eta_R \rvert \lesssim R^{-k}$ for $k=1, ..., j$. With that, we see 
  \begin{align*}
  \sum_{z\in\La}f^{(0)}(z)D_{\boldsymbol{\rho}}\G_k(\ell-z)&=\lim_{R\to \infty}\sum_{z\in\La}\eta_R(z) f^{(0)}(z)D_{\boldsymbol{\rho}}\G_k(\ell-z)\\
  &=\lim_{R\to \infty} -\sum_{z\in\La}\eta_R(z) \nabla^2 V(0)[Du(z), D D_{\boldsymbol{\rho}}\G_k(\ell-z)]\\
  &= - \sum_{z\in\La} \nabla^2 V(0)[Du(z), D D_{\boldsymbol{\rho}}\G_k(\ell-z)],
  \end{align*}
  as the remaining term can be estimated by
  \begin{align*}
  \Big\lvert \sum_{z\in\La} &\nabla^2 V(0)[Du(z), (D_{\boldsymbol{\rho}}\G_k(\ell-z + \sigma) D_\sigma \eta_R(z))_{\sigma}] \Big\rvert\\
   &\lesssim R^{-1} \lVert Du \rVert_{\ell^2} \lVert D_{\boldsymbol{\rho}} \G(\ell- \cdot) \rVert_{\ell^2(B_{2R+C}\setminus B_{R-C})}\\
  &\lesssim R^{1-d} \to 0.
  \end{align*}
  That is, \eqref{greensfunction} holds for all $u \in \HH$ with $\lvert f^{(0)}(\ell) \rvert \lesssim \abs{\ell}^{-\gamma}$. 
\end{proof}

\begin{proof}[Proof of Theorem \ref{thm:structure}]
We now use the Green's function representation \eqref{greensfunction} to estimate $ D_{\boldsymbol{\rho}} u(\ell)$. Let us consider $\lvert f^{(0)} \rvert \lesssim \abs{\ell}^{-d-p} \log^{\alpha} \abs{\ell}$ where we include both the case $\alpha \in \N_0$ and $\alpha < -1$ to include Remark \ref{rem:alphaneg}. Then let us assume there are $p$ vanishing moments.

Although our argument is related to the lower order equivalent in \cite{EOS2016},
it is technically more complex as both the higher derivatives and the higher order divergence form lead to partial summations which, crucially, have to be performed on only specific and separate parts of the lattice. We will therefore  provide full details.

This can be done in a very clean way by splitting the sum \eqref{greensfunction} into four regions; see Figure~\ref{fig:prf_regions} for a visualisation: Region 1 is the far field, where $\lvert z \rvert$ and $\lvert \ell -z\rvert$ are comparable. Region 2 is the intermediary area where $\lvert z \rvert$, $\lvert \ell \rvert$, and $\lvert \ell -z\rvert$ are all comparable. And region 3 and 4 are the areas around $z=0$ and $z=\ell$, where either $\lvert z \rvert$ can be small but $\lvert \ell \rvert$ and $\lvert \ell -z\rvert$ are comparable or $\lvert z-\ell \rvert$ can be small but $\lvert \ell \rvert$ and $\lvert z\rvert$ are comparable. Inserting estimates for the residual $f^{(0)}$ and the lattice Green's function $\G$ and using this split of the sum indeed gives sharp estimates in absolute value. However, as we will show below, this is only a good estimate if both $p=0$ and $j=1,2$. If either $p \geq 1$ or $j \geq 3$ then the sum in \eqref{greensfunction} exhibits significant large scale cancellation effects. To get sharp estimates in these cases, we will remove these cancelling terms via separate partial summations in region 3 and 4. The required discrete derivatives are directly available in the case $j \geq 3$ or are obtained with the help of Proposition \ref{divlemmahigherorder}. To avoid discrete boundary terms, we will not split the four regions sharply but use smooth cutoff functions. The boundary terms in the partial summation are then spread out and can be treated like terms in region 2. 

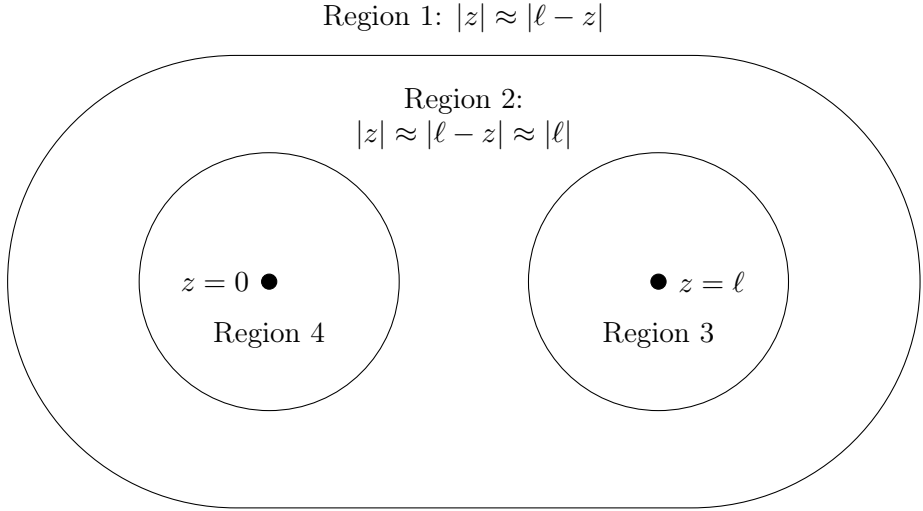
\begin{figure}
  \begin{tikzpicture}
    \draw[rounded corners=3cm] (-6, -3) rectangle ++(12, 6);
    \draw (2.56,0) circle (1.71cm);
    \draw (-2.56,0) circle (1.71cm);
    \draw[fill=black] (2.56, 0)  circle (0.1cm) node[anchor=west] {~$z = \ell$};
    \draw[fill=black] (-2.56, 0)  circle (0.1cm) node[anchor=east] {$z = 0~$};
    \draw (0,3.5) node {Region 1: $|z| \approx |\ell-z|$};
    \draw (0,2.4) node {Region 2: };
    \draw (0,1.95) node { $|z| \approx |\ell-z| \approx |\ell|$};
    \draw (2.56, -0.7) node {Region 3};
    \draw (-2.56, -0.7) node {Region 4};
\end{tikzpicture}
  \label{fig:prf_regions}
  \caption{Splitting $\mathbb{R}^d$ into four regions as used in the proof of Theorem \ref{thm:structure}.}
\end{figure}

So, let us take a smooth cutoff $\eta_\ell \colon \R^d \to [0,1]$ satisfying $\eta_\ell(z)=1$ for $\lvert z \rvert \leq \lvert \ell \rvert/4$ and $\eta_\ell(z)=0$ for $\lvert z \rvert \geq \lvert \ell \rvert/2$, as well as $\lvert \nabla^k \eta_\ell \rvert \lesssim \lvert \ell \rvert^{-k}$ for $k=1, ..., j$. As discussed, we now split the sum in \eqref{greensfunction} according to the four regions.
\begin{align}
\sum_{z\in\La}f^{(0)}(z)D_{\boldsymbol{\rho}}\G_k(\ell-z) &= \sum_{\lvert z \rvert > 2 \lvert \ell \rvert }f^{(0)}(z) D_{\boldsymbol{\rho}} \G_k(\ell-z)\nonumber\\
&\hspace{-1cm} + \sum_{\lvert z \rvert \leq 2 \lvert \ell \rvert}f^{(0)}(z)D_{\boldsymbol{\rho}}\G_k(\ell-z)(1-\eta_\ell(z))(1-\eta_\ell(\ell-z))\nonumber\\
&\hspace{-1cm} + \sum_{z\in\La}f^{(0)}(z)D_{\boldsymbol{\rho}}\G_k(\ell-z)\eta_\ell(z) \nonumber\\ 
&\hspace{-1cm} + \sum_{z\in\La}f^{(0)}(z)D_{\boldsymbol{\rho}}\G_k(\ell-z)\eta_\ell(\ell-z) \nonumber \\
& =: T_1 + T_2 + T_3 + T_4 \label{eq:T1-4definition}
\end{align}

We first estimate the far-field term
\begin{align*}
\lvert T_1 \rvert \lesssim  \sum_{\overset{z\in\La}{\lvert z \rvert > 2 \lvert \ell \rvert}} \abs{z}^{(-d-p) +  (2-d-j)} \log^{\alpha} \abs{z} \lesssim\abs{\ell}^{2-j-d-p} \log^{\alpha} \abs{\ell},
\end{align*}
where the logarithmic term is estimated trivially by $\log^{\alpha} \abs{z} \leq \log^{\alpha} \abs{\ell}$ for negative $\alpha$, while for $\alpha \in \N_0$ the estimate instead follows from partial integrations of the resulting one-dimensional radial integral $\int_{ \lvert \ell \rvert}^\infty r^{1-j-d-p} \log^{\alpha} r \,dr$.

The intermediary area is even more direct as we can just estimate the functions uniformly and multiply by the number of lattice points in the area
\begin{align*}
\lvert T_2 \rvert &\lesssim   \sum_{\overset{z\in\La}{\overset{\frac{1}{2}\lvert \ell \rvert \leq\lvert z \rvert \leq 2 \lvert \ell \rvert}{\lvert \ell-z \rvert \geq \frac{1}{2} \lvert \ell \rvert}}} \lvert f^{(0)}(z)\vert \lvert D_{\boldsymbol{\rho}}\G_k(\ell-z) \rvert\\
&\lesssim \abs{\ell}^d \abs{\ell}^{-d-p} \log^{\alpha} \abs{\ell} \abs{\ell}^{2-d-j}\\
&= \abs{\ell}^{2-j-d-p} \log^{\alpha} \abs{\ell}.
\end{align*}

Next, $T_3$ can be estimated by
\begin{align*}
\lvert T_3 \rvert &= \Big\lvert \sum_{z\in\La}\eta_\ell(z)f^{(0)}(z)D_{\boldsymbol{\rho}}\G(\ell-z)\Big\rvert \\
&\lesssim \abs{\ell}^{2-d-j}\sum_{\lvert z \rvert \leq \lvert \ell \rvert} \abs{z}^{-d-p} \log^{\alpha} \abs{z}.
\end{align*}
Hence, we have $\lvert T_3 \rvert \lesssim \abs{\ell}^{2-d-j}$ if either $p>0$ or $\alpha <-1$. If, on the other hand, $p=0$ and $\alpha \in \N_0$, then we obtain $\lvert T_3 \rvert \lesssim \abs{\ell}^{2-d-j} \log^{\alpha+1}\lvert \ell \rvert$.

Finally, for $T_4$ the estimate is 
\begin{align*}
\lvert T_4 \rvert &= \sum_{z\in\La}f^{(0)}(z)D_{\boldsymbol{\rho}}\G_k(\ell-z)\eta_\ell(\ell-z) \\
&\lesssim \abs{\ell}^{-d-p} \log^{\alpha} \abs{\ell}
  {\sum_{\lvert \ell-z \rvert \leq \lvert \ell \rvert/2}} \abs{\ell-z}^{2-d-j}\\
&\lesssim \abs{\ell}^{-d-p} \log^{\alpha} \abs{\ell} (1+\delta_{j 1} \lvert \ell \rvert +\delta_{j 2} \log\abs{\ell}).
\end{align*}
Putting all four estimates together, this completes the proof in the special case where both $p=0$ and $j =1,2$. 

For $p \geq 1$, we need a better estimate on $T_3$. We choose $f^{(m)}$ according to Proposition \ref{divlemmahigherorder}. We claim that for $0 \leq m \leq p$
\begin{align}
\Big\lvert \sum_{z\in\La}&\eta_\ell(z)f^{(0)}(z)D_{\boldsymbol{\rho}}\G(\ell-z)\Big\rvert \nonumber \\
 &\lesssim \abs{\ell}^{2-j-d-p} \log^{\alpha} \abs{\ell} + \Big\lvert \sum_{z\in\La} \sum_{\boldsymbol{\sigma} \in \mathcal{S}^m}\eta_\ell(z)f^{(m)}_{\boldsymbol{\sigma}}(z)D_{\boldsymbol{\sigma}} D_{\boldsymbol{\rho}}\G(\ell-z)\Big\rvert. \label{greensfunctionpartialsummation}
\end{align}
Let us prove \eqref{greensfunctionpartialsummation} by induction over $m$. Clearly, it is true for $m=0$ since the second term on the right-hand side is identical to the left-hand side. Given its validity for a $m$, with $m+1 \leq p$, we now employ Proposition \ref{divlemmahigherorder} and summation by parts to obtain 
\begin{align*}
\Big\lvert \sum_{z\in\La}&\eta_\ell(z)f^{(0)}(z)D_{\boldsymbol{\rho}}\G(\ell-z)\Big\rvert \\
 &\lesssim \abs{\ell}^{2-j-d-p} \log^{\alpha} \abs{\ell} + \Big\lvert \sum_{z\in\La} \sum_{\boldsymbol{\sigma} \in \mathcal{S}^m}\eta_\ell(z)f^{(m)}_{\boldsymbol{\sigma}}(z)D_{\boldsymbol{\sigma}} D_{\boldsymbol{\rho}}\G(\ell-z)\Big\rvert\\
 &= \abs{\ell}^{2-j-d-p} \log^{\alpha} \abs{\ell}+ \Big\lvert \sum_{z\in\La} \sum_{\boldsymbol{\sigma} \in \mathcal{S}^m}\eta_\ell(z) \sym f^{(m)}_{\boldsymbol{\sigma}}(z)D_{\boldsymbol{\sigma}} D_{\boldsymbol{\rho}}\G(\ell-z)\Big\rvert\\
 &= \abs{\ell}^{2-j-d-p} \log^{\alpha} \abs{\ell} + \Big\lvert \sum_{z\in\La} \sum_{\boldsymbol{\sigma}}\eta_\ell(z) {\rm Div}_{\mathcal{S}} f^{(m+1)}_{\boldsymbol{\sigma}}(z)D_{\boldsymbol{\sigma}} D_{\boldsymbol{\rho}}\G(\ell-z)\Big\rvert\\
 &\leq \abs{\ell}^{2-j-d-p} \log^{\alpha} \abs{\ell} + \Big\lvert \sum_{z\in\La} \sum_{\tau \in \mathcal{S}}\sum_{\boldsymbol{\sigma} \in \mathcal{S}^m}\eta_\ell(z)f^{(m+1)}_{\boldsymbol{\sigma} \tau}(z)D_{\tau} D_{\boldsymbol{\sigma}} D_{\boldsymbol{\rho}}\G(\ell-z)\Big\rvert\\
 &\quad +  \sum_{z\in\La} \sum_{\tau \in \mathcal{S}}\sum_{\boldsymbol{\sigma} \in \mathcal{S}^m} \big\lvert D_{\tau}\eta_\ell(z)f^{(m+1)}_{\boldsymbol{\sigma} \tau}(z) D_{\boldsymbol{\sigma}} D_{\boldsymbol{\rho}}\G(\ell-z+\tau) \big\rvert.
\end{align*}
The last term is concentrated in the annulus where $D_\tau\eta$ is non-zero and can be estimated by 
\begin{align*}
\sum_{z\in\La} \sum_{\tau \in \mathcal{S}}\sum_{\boldsymbol{\sigma} \in \mathcal{S}^m} &\big\lvert D_{\tau}\eta_\ell(z)f^{(m+1)}_{\boldsymbol{\sigma} \tau}(z) D_{\boldsymbol{\sigma}} D_{\boldsymbol{\rho}}\G(\ell-z+\tau) \big\rvert\\
 &\lesssim \abs{\ell}^d \abs{\ell}^{-1}\abs{\ell}^{-d-p+m+1}\log^{\alpha} \abs{\ell} \abs{\ell}^{2-d-m-j}\\
 &=\abs{\ell}^{2-j-d-p}\log^{\alpha} \abs{\ell},
\end{align*}
which completes the proof of \eqref{greensfunctionpartialsummation}.

Using \eqref{greensfunctionpartialsummation} for $k=p$, we find
\begin{align*}
\lvert T_3 \rvert &= \Big\lvert \sum_{z\in\La}\eta_\ell(z)f^{(0)}(z)D_{\boldsymbol{\rho}}\G(\ell-z)\Big\rvert \\
&\lesssim \abs{\ell}^{2-j-d-p} \log^{\alpha} \abs{\ell} +  \abs{\ell}^{2-d-j-p}\sum_{\lvert z \rvert \leq \lvert \ell \rvert} \abs{z}^{-d} \log^{\alpha} \abs{z}.
\end{align*}
Therefore, $\lvert T_3 \rvert \lesssim \abs{\ell}^{2-j-d-p}$ for $\alpha < -1$ and $\lvert T_3 \rvert \lesssim \abs{\ell}^{2-j-d-p} \log^{\alpha+1} \abs{\ell}$ for $\alpha \in \N_0$. This finishes the proof for $j=1,2$ and arbitrary $p$. 

Finally, for $j \geq 3$, we need a better estimate for $T_4$. In this case, let us split the difference directions as $\boldsymbol{\rho}=(\boldsymbol{\sigma},\boldsymbol{\tau})$, with $\boldsymbol{\sigma} \in \Rc^2$. Then, 
\begin{align*}
\sum_{z\in\La}f^{(0)}(z)D_{\boldsymbol{\rho}}\G(\ell-z)\eta_\ell(\ell-z)  &=  \sum_{z\in\La}  D_{-\boldsymbol{\tau}}\Big(f^{(0)}(z)\eta_\ell(\ell-z) \Big)D_{\boldsymbol{\sigma}}\G(\ell-z).
\end{align*}
Note that if at least one discrete derivative falls on $\eta_\ell$ the estimate is similar to $T_2$ as $\lvert z \rvert$, $\lvert \ell \rvert$, $\lvert \ell - z \rvert$ are then comparable for all non-zero terms. We thus get
\begin{align*}
\lvert T_4 \rvert &\leq \sum_{z\in\La}  \Big\lvert D_{-\boldsymbol{\tau}}\Big(f^{(0)}(z)\eta_\ell(\ell-z) \Big) \Big\rvert \lvert D_{\boldsymbol{\sigma}}\G(\ell-z) \rvert\\
&\lesssim  \sum_{z\in\La} \sum_{m=0}^{j-2}  \lvert D_{-\mathcal{S}}^{m}f^{(0)}(z)\rvert \lvert D_{-\mathcal{S}}^{j-2-m} \eta_\ell(\ell-z) \rvert  \lvert D_{\boldsymbol{\sigma}}\G(\ell-z) \rvert\\
&\lesssim  \sum_{m=0}^{j-3} \abs{\ell}^d \abs{\ell}^{-d-p-m}\log^{\alpha}\abs{\ell} \abs{\ell}^{2-j+m}  \abs{\ell}^{-d} \\
& \qquad + \sum_{z\in\La}  \lvert D_{-\mathcal{S}}^{j-2}f^{(0)}(z)\rvert \lvert \eta_\ell(\ell-z) \rvert  \lvert D_{\boldsymbol{\sigma}}\G(\ell-z) \rvert\\
&\lesssim \abs{\ell}^{2-d-p-j}\log^{\alpha}\abs{\ell} + \abs{\ell}^{2-d-p-j}\log^{\alpha}\abs{\ell} \sum_{\lvert z \rvert \leq \lvert \ell \rvert}  \abs{z}^{-d}\\
&\lesssim \abs{\ell}^{2-d-p-j}\log^{\alpha+1}\abs{\ell}.
\end{align*}
Note that we freely used the estimates of $\lvert D_{\mathcal{S}}^{j-2}f^{(0)} \rvert $ for $ \lvert D_{-\mathcal{S}}^{j-2}f^{(0)} \rvert $, as $\mathcal{S}$ spans the lattice and thus these estimates are equivalent.
\end{proof}

Theorem~\ref{thm:structurewithmoments} is almost an immediate consequence of 
Theorem~\ref{thm:structure}. The only additional ingredient we need for the 
proof is the following lemma.

\begin{lemma} \label{lem:bijectionmomentstob}
In the setting of Theorem \ref{thm:structurewithmoments} define
a function $v_b$ parameterised by $b = (b^{(i,k)})_{i,k}$ by 
\[v_b := \sum_{i=0}^{p-1} \sum_{k=1}^n b^{(i,k)} : D_{\rm \mathcal{S}}^i \G_k.\]
Then there are unique $b^{(i,k)} \in (\R^{\mathcal{S}})^{\odot i}$, $i=1, \dots, p-1$, $k=1,\dots, N$, such that $\mathcal{I}_j(u)=\mathcal{I}_j(v_b)$ for all $j=0, \dots, p-1$.
\end{lemma}
\begin{proof}
For any $b$ one calculates
\begin{align*}
\mathcal{I}_j(v_b) &= \sum_{\ell} H[v_b](\ell) \otimes \ell^{\otimes j} \\
&= \sum_{i=0}^{p-1} \sum_{k=1}^N \sum_{\boldsymbol{\rho} \in \mathcal{S}^i} b^{(i,k)}_{\boldsymbol{\rho}}\sum_{\ell} H [D_{\boldsymbol{\rho}} \G_k](\ell) \otimes \ell^{\otimes j}.
\end{align*}
Note that $H [D_{\boldsymbol{\rho}} \G_k] = D_{\boldsymbol{\rho}} \delta_0 e_k$. In particular, there are no summability problems as this expression has compact support. With that in mind we can use Lemma \ref{th:Diffpellp_identity} to calculate

\begin{align*}
(\mathcal{I}_j(v_b))_{k\cdot} &=  \sum_{i=0}^{p-1} \sum_{\boldsymbol{\rho} \in \mathcal{S}^i} b^{(i,k)}_{\boldsymbol{\rho}}D_{\boldsymbol{-\rho}} ( \ell^{\otimes j})(0)\\
&= (-1)^j j!  \sum_{\boldsymbol{\rho} \in \mathcal{S}^j} b^{(j,k)}_{\boldsymbol{\rho}} \boldsymbol{\rho}^\odot.
\end{align*}

Therefore, it is sufficient to show that the linear map
\[T \colon (\R^{\mathcal{S}})^{\odot j} \to (\R^d)^{\odot j}\]
given by
\[Tb = \sum_{\boldsymbol{\rho} \in \mathcal{S}^j} b_{\boldsymbol{\rho}} \boldsymbol{\rho}^\odot\]
is a bijection. By dimensionality (recall that $\mathcal{S}$ is chosen to be a basis), this is equivalent to $T$ being one--to--one which follows from Lemma \ref{symmetrictensors}.
\end{proof}

\begin{proof}[Proof of Theorem \ref{thm:structurewithmoments}]
Let $v_b=\sum_{i=0}^{p-1} \sum_{k=1}^n b^{(i,k)} : D_{\rm \mathcal{S}}^i \mathcal{G}_k$, and choose $b$ according to Lemma \ref{lem:bijectionmomentstob}. As the moments are linear, the result then follows directly from Theorem \ref{thm:structure}.
\end{proof}

Note that even though the choice of $b$ in Lemma \ref{lem:bijectionmomentstob} is unique, that does not necessarily mean that the choice of $b$ for Theorem \ref{thm:structurewithmoments} is unique. Indeed, it can happen that for certain $b$ the sum $ \sum_{k=1}^N b^{(i,k)} : D_{\rm \mathcal{S}}^i \mathcal{G}_k$ exhibits sufficient cancellation to be part of the error estimate, for example if the coefficients $b$ correspond to a discrete approximation of the CLE equation.

\paragraph{Decay estimates for the far-field continuum equations}

We also want to have decay estimates for the equations \eqref{eq:correctorequation} introduced in \S\,\ref{sec:contdev}. These can be obtained along similar lines to the discrete case. However, we are purely interested in the far-field which simplifies things a bit and avoids a few additional complications specific to the continuum case.

\begin{proposition} \label{prop:contestimates}
Let $p, \alpha \in \N_0$. For any $f \in C^{\infty}(\R^d \backslash B_{R_0/2}(0);\R^N)$ with
\[\lvert \nabla^j f(\ell) \rvert \leq C_j \lvert \ell \rvert^{-d-p-j} \log^{\alpha} \lvert \ell \rvert \quad \forall \ell \in  \R^d \backslash B_{R_0/2}(0)\]
for all $j\in\N_0$, there exists a $u \in C^{\infty}(\R^d \backslash B_{R_0}(0);\R^N)$ and $\tilde{C}_j$ with
\[\lvert \nabla^j u(\ell) \rvert \leq \tilde{C}_j \lvert \ell \rvert^{2-d-p-j} \log^{\alpha+1} \lvert \ell \rvert \quad \forall \ell \in  \R^d \backslash B_{R_0}(0) \]
for all $j\in\N$, such that
\[H^{\rm C}[u]=f \quad \text{in}\ \R^d \backslash B_{R_0}(0).\]
\end{proposition}
\begin{remark}
It is important to point out that the additional logarithm in the result is a generic aspect of the problem and not just a limitation of our proof. To see that consider $\alpha=0$ and $p=1$. Specifically, let us take a $f$ with $f = \divo g$, $\lvert g(\ell) \rvert \lesssim \lvert \ell \rvert^{-d}$, and $\lvert f(\ell) \rvert \lesssim \lvert \ell \rvert^{-d-1}$. Then our proof below actually shows
\begin{align*}
\partial_i u_k &= O(\lvert \ell \rvert^{-d}) + \nabla  \partial_i (G_0)_k(\ell) \colon \int_{B_{\lvert \ell \rvert/4}(0) \backslash B_1(0)}  g(z) \,dz.
\end{align*}
Even for something as simple as $g(\ell) : = E_{11} \lvert \ell \rvert^{-d}$  one directly gets a logarithm from the integral. That means the logarithm term naturally comes from a summing of the stresses (for $p=1$) around the defect core.

Of course there are special cases where there is no logarithmic term appearing. For $p=1, \alpha=0$ that would be any setting with enough cancellation such that \[ \sup_{R>1} \Big\lvert \int_{B_{R}(0) \backslash B_1(0)}  g(z) \,dz \Big\rvert < \infty.\]
For example, the first corrector equation in \cite{BBO19} satisfies $\int_{\partial B_R(0)} g(z) \,dz =0$ and indeed there was no logarithmic term.
\end{remark}

\begin{proof}[Sketch of the Proof] As the equation only needs to be satisfied outside of $B_{R_0}$, we can assume without loss of generality that $f \in C^{\infty}(\R^d;\R^N)$, by multiplying $f$ with an appropriate cutoff function. Our ultimate goal will be to define $u := f \ast G^{\rm C}$ and follow the same estimates as in the discrete case. The main problem in adapting the discrete proof to the continuum is the need to replicate Lemma~\ref{divlemma} and Lemma~\ref{divlemmahigherorder} which allow conversion to appropriate divergence form for $p \geq 1$, and we sketch out the route to the analogues of these results here.

  We follow the general idea of the proof in \cite{EOS2016}, defining $f_0:=f$, $f_{n+1}(\ell):=3^d f_n(3\ell)$, and
\[(g_n)_i:= 3^{i-1}\int_{\ell_i}^{3 \ell_i} f_n(3\ell_1, \dots, 3 \ell_{i-1}, \lambda, \ell_{i+1}, \dots,\ell_{d}) \,d\lambda. \]
By direct computation it can be checked that
\begin{equation}\label{eq:gnrecurrence}
f_{n+1} = f_n + \divo g_{n},
\end{equation}
and using the bound on $f$ assumed in the statement and $p \geq 1$, one concludes that $f_n \to 0$ uniformly in any $\R^d \backslash B_{\delta}(0)$. Similarly, we may show that the sequence $g_n$ is uniformly summable over the same sets, and we define $g:= \sum_{n=0}^{\infty} g_n$. Using the relation \eqref{eq:gnrecurrence}, we find $f = -\divo g$ in $\R^d \backslash \{0\}$. It also easy to check that $g \in C^\infty(\R^d \backslash \{0\})$ and that $g$ satisfies the decay $\lvert g(\ell) \rvert \lesssim \lvert \ell \rvert^{-d-p+1} \log^{\alpha} \lvert \ell \rvert$ away from zero. In the continuum setting, a possible consequence of this construction is that $g$ may have a singularity at $\ell =0$. However, since we are only interested in the far-field behaviour, we can simply remove any singularity by redefining $g$ on the interior of $B_{R_0}(0)$ such that $g \in C^\infty(\R^d; \R^{N \times d})$, and then redefine $f:= \divo g$. This operation does not affect the behaviour outside $B_{R_0}(0)$, and we have therefore recovered the equivalent of Lemma~\ref{divlemma}.

A similar construction and inductive argument to that used in the proof of Lemma~\ref{divlemmahigherorder} allows us to generate an extension of $f$ to $f=f^{(0)} \in C^\infty(\R^d; \R^N)$ such that there are tensor fields $f^{(k)} \in C^\infty(\R^d; \R^N \otimes (\R^d)^{\otimes k})$ for $1 \leq k \leq p$, such that $\sym f^{(k-1)}= -\divo f^{(k)}$ for $1 \leq k \leq p$ satisfying the decay estimate $\lvert f^{(k)}(\ell) \rvert \lesssim \lvert \ell \rvert^{-d-p+k} \log^{\alpha} \lvert \ell \rvert$ for $\ell$ bounded away from zero. To avoid generating a singularity at $\ell=0$ in this case, again we redefine the functions in $B_{R_0}(0)$ starting at the highest order $f^{(p)}$.

With this ability to transform $f$ to (higher-order) divergence form, one can now follow the proof of Theorem \ref{thm:structure} almost verbatim to estimate the derivatives of $u=f \ast G^{\rm C}$ using the split into the four different contributions defined in \eqref{eq:T1-4definition}. Note that the singularity of $G^{\rm C}$ is not a problem as $G^{\rm C}$ and its first derivative are locally integrable everywhere, and we only need to take further derivatives of the Green's function in the estimate of $T_3$, which is the part away from that singularity where $G^{\rm C}$ is smooth. Indeed, in the estimate of $T_4$, we can use the fact that we have assumed estimates for derivatives of $f$ of arbitrary order and only put one derivative on the Green's function instead of two.
\end{proof}

\section{The lattice Green's function and series expansion} \label{sec:greensfunctions}
In this section, we define the lattice Green's function, and construct a series expansion in terms of continuum kernels proving Theorem \ref{thm:greensfunctionexpansion}. The general strategy is to expand the inverse of the discrete Fourier multiplier corresponding to the lattice operator $H$ as a series of homogeneous terms about $k=0$, and then use tools developed by Morrey in \cite{m66} to provide integral representations of these terms in real space which allow us to explicitly characterise the decay and even homogeneity of each term. In the first parts we will also follow some ideas that were developed in \cite{MR02} for the simpler scalar case where $N=1$. 

\subsection{Fourier transforms}
We begin by providing definitions of the Fourier transforms we deploy here. For $f:\La\to\R^n$, we define the Fourier transform
$\mathcal{F}_a[f]:\mathcal{B}\to\R^n$, where $\mathcal B$ is the Brillouin zone
of the lattice, via
\begin{equation*}
  \mathcal{F}_a[f](k) = \sum_{\ell\in\La}f(\ell)\mathrm{e}^{-ik\cdot \ell}.
\end{equation*}
The inverse is given by
\begin{equation*}
  \mathcal{F}_a^{-1}[g](\ell) = \frac{1}{|\mathcal B|}\int_{\mathcal B} g(k)\mathrm e^{ik\cdot \ell}\, \mathrm dk.
\end{equation*}
Note that $\lvert \mathcal B \rvert = \frac{(2 \pi)^d}{ c_{\rm vol}}$, where $c_{\rm vol}$ is the volume of a fundamental cell of $\Lambda$.

We note the following properties.
\begin{itemize}
\item If $\boldsymbol{\rho} = (\rho_1,\ldots,\rho_j)\in\mathcal R^j$, then
  \begin{equation*}\mathcal{F}_a[D_{\boldsymbol{\rho}}f](k)
  = \prod_{n=1}^j (\mathrm e^{i k\cdot \rho_n}-1) \mathcal{F}_a[f](k).
\end{equation*}
\item For each $m\in\N$, we have that
  \begin{equation*}
    \sum_{\ell\in\La} f(\ell)\otimes \ell^{\otimes m} = i^m \nabla^m \mathcal{F}_a[f](0).
  \end{equation*}
  if $f$ decays fast enough such that the sum converges absolutely.
\end{itemize}

We also use the standard continuum Fourier transform
\begin{equation*}
  \mathcal{F}_c[f](k) = \int_{\R^d} f(x)\mathrm{e}^{-ik\cdot x} \,dx.
\end{equation*}
defined for $f \in L^1(\R^d; \R^N)$ and extended to the space of tempered distributions $\mathcal{S}'(\R^d; \R^N)$ by duality.
As usual the inverse is given by
\begin{equation*}
  \mathcal{F}_c^{-1}[g](x) = \frac{1}{(2 \pi)^d} \int_{\R^d} g(k)\mathrm{e}^{ik\cdot x} \,dx.
\end{equation*}

\subsection{Fourier multiplier series manipulation}\label{sec:four-mult-seri}
Next, we formally carry out a series development of the inverse of the discrete lattice operator Fourier multiplier, which we will subsequently analyse in detail.

As before consider the Hamiltonian $H=\delta^2 \mathcal{E}(0)$ given pointwise by
  $H [u]= -\Div \big(\sum_{\rho\in\mathcal R} \CC_{\sigma \rho} D_\rho u\big)_{\sigma \in \Rc}$, where $\CC = \nabla^2 V(0) \in \R^{(\Rc \times N) \times (\Rc \times N)}$ and we write $\CC_{\sigma \rho} = (\CC_{\sigma i\rho j})_{ij}\in \R^{N \times N}$. $\CC$ satisfies the symmetry $\CC_{i \rho k \sigma} = \CC_{k \sigma i \rho}$ as a second derivative and $\CC_{i \rho k \sigma} = \CC_{i (-\rho) k (-\sigma)}$ due to equation \eqref{eq:pointsymmetry}. Then (see \cite{HO2012,braun16static}) $H$ has a representation as a Fourier multiplier
  \begin{align*}
    \widehat{H}(k) &:=\mathcal{F}_a[H](k)\\
    & = \sum_{\rho, \sigma \in\Rc} \CC_{\rho \sigma} \big(2\sin^2(\tfrac12 k\cdot \rho) + 2\sin^2(\tfrac12 k\cdot \sigma) - 2\sin^2(\tfrac12 k\cdot (\rho-\sigma)) \big)\\
    & = \sum_{\rho \in (\Rc \cup \{0\})- (\Rc \cup \{0\})  } 4 A_{\rho} \sin^2(\tfrac12 k\cdot \rho),
  \end{align*}
  where the matrices $A_\rho$ are implicitly defined through the second identity.

As before we assume lattice stability \eqref{eq:results:latticestab}. In reciprocal space this entails that there exists $c_0>0$
such that
\begin{equation*}
  \widehat{H}(k)\geq c_0 \sum_{\rho\in\mathcal{R}}4\sin^2(\tfrac12 k\cdot\rho)\Id,\quad\text{for all }k\in\mathcal B,
\end{equation*}
where $\Id$ is the $N\times N$ identity matrix. We note that $\widehat{H}\in C^\infty_{\text{per}}(\mathcal B)$, and by virtue of lattice ellipticity, $\widehat{H}$ is strictly positive definite except when $k=0$. This entails that the matrix inverse $\widehat{H}^{-1}(k)$ exists everywhere in $\mathcal B$ except at $k=0$, and moreover,
\begin{equation*}
  \widehat{H}^{-1}(k)\in C^\infty_{\text{per}}(\mathcal B\setminus \{0\}).
\end{equation*}

We define
\begin{equation*}
  \widehat{H}_{2n}(k) := \sum_{\rho \in (\Rc \cup \{0\})- (\Rc \cup \{0\})} \frac{(-1)^{n+1}2(k\cdot \rho)^{2n}}{(2n)!} A_{\rho},
\end{equation*}
which is $2n$ homogeneous in $k$. Indeed, as $4\sin^2(\frac12x)=2(1-\cos(x))$ we have the globally convergent power series
\begin{equation*}
  \widehat{H}=\sum_{n=1}^\infty \widehat{\mathcal H}_{2n}(k).
\end{equation*}
A standard procedure can not be used to invert this series. Defining $S = \widehat{H}-\widehat{H}_2$, it is straightforward to check that
$|\widehat{S}(k)| \lesssim |k|^4$ for $k$ in a
neighbourhood of $0$. Also, since
$c_0 |k|^2\Id\leq \widehat{H}_2(k)$ for any $k$,
we see that $|\widehat{H}_2^{-1/2}(k)|\lesssim |k|^{-1}$.

This allows us to write
\begin{equation*}
  \widehat{ H}^{-1}(k) = \widehat{H}_2^{-1/2}\big(\Id+\widehat{H}_2^{-1/2}\widehat{S}\widehat{ H}_2^{-1/2}\big)^{-1}\widehat{H}_2^{-1/2},
\end{equation*}
for all $k \neq 0$ sufficiently small as we have
\begin{equation*}
  |\widehat{H}_2^{-1/2}\widehat{S}\widehat{H}_2^{-1/2}|\lesssim |k|^2.
\end{equation*}
In a neighbourhood of $0$, we thus have the absolutely convergent series representation
\begin{align*}
  \widehat{H}^{-1}
  &= \widehat{H}_2^{-1/2}\bigg( \sum_{j=0}^\infty (-1)^j
    \Big(\widehat{H}_2^{-1/2}\widehat{S}\widehat{H}_2^{-1/2}\Big)^j
    \bigg)\widehat{H}_2^{-1/2}\\
  &=\widehat{H}_2^{-1}- \widehat{H}_2^{-1}\widehat{S}\widehat {H}_2^{-1} +
    {H}_2^{-1}\widehat{S}\widehat {H}_2^{-1}\widehat{S}\widehat {H}_2^{-1}
    -\cdots.
\end{align*}
Now, by inserting 
\[ 
  \widehat{S} = \widehat{H} - \widehat{H}_2
  = \sum_{n \geq 2} \widehat{H}_{2n},
\] 
and re-summing, we may express $H^{-1}$ as a series of terms with increasing
homogeneity:
\begin{equation}\label{eq:HinvSeries}
\begin{aligned}
  \widehat{H}^{-1}
  &= \underbrace{\widehat{H}_2^{-1}}_{=:\mathcal A_{-2}}
  +\,\underbrace{-\widehat{ H}_2^{-1}{\widehat{H}_4\widehat {H}_2^{-1}}}_{=:\mathcal{A}_0}\\
  &\qquad\quad
    +\underbrace{\big(\widehat{H}_2^{-1}\widehat{H}_4\widehat{H}_2^{-1}\widehat{H}_4\widehat{H}_2^{-1}
    -\widehat{\mathcal H}^{-1}_2\widehat{H}_6\widehat{H}^{-1}_2\big)}_{=:\mathcal A_2}+\cdots\\
  &=\sum_{n=-1}^\infty \mathcal{A}_{2n}(k).
\end{aligned}
\end{equation}
Explicitly, $\mathcal{A}_{2n}(k)$ is defined by the finite sum
\begin{equation}\label{eq:A2n_def}
\begin{aligned}
\mathcal{A}_{2n}(k) &= \sum_{j=1}^{n+1} \sum_{\boldsymbol{\alpha} \in \mathbf{A}_{j,2n}} \widehat{H}_2^{-1}\widehat{H}_{\boldsymbol{\alpha}_1}\widehat{H}_2^{-1}\cdots \widehat{H}_2^{-1}\widehat{H}_{\boldsymbol{\alpha}_j}\widehat{H}_2^{-1}, \quad \text{where}\\
\mathbf{A}_{j,2n} &= \{ \boldsymbol{\alpha} \in (2\N)^j : \boldsymbol{\alpha}_i \geq 4\ \forall i\ \text{and}\  \boldsymbol{\alpha}_1+ \dots \boldsymbol{\alpha}_j -2j-2 = 2n \}.
\end{aligned}
\end{equation}
This means $\mathcal{A}_{2n}$ is $2n$-homogeneous.

In cases where the first term in the expansion of $H_2$ is a multiple of the identity,
i.e. $\widehat{H}_2(k)=h_2(k)\Id$, the series simplifies to
\begin{equation*}
  H^{-1} = h_2^{-1}\Id-h_2^{-2}\widehat{H}_4
    +h_2^{-3}\big(\widehat{H}_4^2
    -h_2\widehat{H}_6\big)+\cdots,
\end{equation*}
and yet further simplification can be made if $H$ is a multiple of the identity matrix, e.g., in the scalar setting $N=1$, which is the case considered in \cite{MR02}.

\subsection{The lattice Green's function and its development at infinity}
\label{sec:latGr_develop}
With the previous series development complete, we may introduce and analyse the lattice Green's function, which acts as an inverse to the lattice operator $H$.

If $d\geq 3$, we define the lattice Green's function to be
\begin{align*}
\mathcal{G}(l):=\frac{1}{\lvert \mathcal{B} \rvert}\int_{\mathcal{B}} \hat{H}^{-1}(k) e^{ik \ell} \, dk.
\end{align*}
and if $d=2$, we set
\[\mathcal{G}(\ell) :=\lim_{\delta \to 0} \Big( \frac{1}{\lvert \mathcal{B} \rvert}\int_{\mathcal{B} \backslash B_{\delta}(0)} \hat{H}^{-1}(k) e^{ik \ell} \, dk + (\log \delta + \gamma) \frac{1}{\lvert \mathcal{B} \rvert} \int_{\mathbb{S}^1} \mathcal{A}_{-2}(\sigma)  \, d\sigma  \Big) \]
where $\gamma$ is the Euler–Mascheroni constant. To demonstrate that the limit in the latter definition exists, we take $0 < \delta < \eps$ and calculate
\begin{align*}
\hspace{1.5cm} &\hspace{-1.5cm} \frac{1}{\lvert \mathcal{B} \rvert}\int_{B_{\eps}(0) \backslash B_{\delta}(0)} \hat{H}^{-1}(k) e^{ik \ell} \, dk\\
 &= O(\eps^2) + \frac{1}{\lvert \mathcal{B} \rvert}\int_{B_{\eps}(0) \backslash B_{\delta}(0)} \mathcal{A}_{-2}(k) e^{ik\ell} \, dk \\
  &= O(\eps^2) + \frac{1}{\lvert \mathcal{B} \rvert}\int_{\delta}^{\eps} \int_{\mathbb{S}^1}r^{-1} \mathcal{A}_{-2}(\sigma) e^{i\sigma \ell r} \, d\sigma  dr \\
  &= O(\eps^2) + O(\lvert  \ell \rvert \eps) + \frac{\log( \eps) - \log( \delta)}{\lvert \mathcal{B} \rvert} \int_{\mathbb{S}^1} \mathcal{A}_{-2}(\sigma)  \, d\sigma.
\end{align*}
In the above argument, we have used that $\hat{H}^{-1}-\mathcal{A}_{-2}$ is bounded in a neighbourhood of $k=0$, and the fact that $|1-e^{i\sigma \ell r}|=O(|\ell| r)$ for all suitably small $r$. Rearranging, we have established a uniform bound independent of $\delta$, and hence the limit exists.

We now use the series expansion \eqref{eq:HinvSeries} to define a series of functions $G_{n}$ which approximate $\mathcal{G}$. Note that $\mathcal{A}_{2n-2}$ is locally integrable and defines a tempered distribution unless both $n=0$ and $d=2$. Except in this special case, we therefore define
\begin{equation*}
G_{n} := c_{\rm vol} \mathcal{F}_c^{-1} [\mathcal{A}_{2n-2}].
\end{equation*}
In the special case where $n=1$ and $d=2$ we define
\begin{align*} 
  G_0 &:=  \lim_{\delta \to 0} \bigg( 
      \frac{c_{\rm vol}}{(2 \pi)^2}\int_{\R^2 \backslash B_{\delta}(0)} \mathcal{A}_{-2}(k) e^{ik \ell} \, dk  \\ 
    & \hspace{1.5cm}
    + (\log \delta + \gamma) \frac{c_{\rm vol}}{(2 \pi)^2} \int_{\mathbb{S}^1} \mathcal{A}_{-2}(\sigma)  \, d\sigma  \bigg)
\end{align*}
in the distributional sense, i.e.
\begin{align*} 
    \langle G_0, \varphi \rangle 
    &=  \lim_{\delta \to 0} \bigg( \frac{c_{\rm vol}}{(2 \pi)^2} \int_{\R^2 \backslash B_{\delta}(0)} \mathcal{F} [\varphi](k) \mathcal{A}_{-2}(k)\,dk  \\ 
    & \hspace{1.5cm} +  c_{\rm vol}\frac{\log \delta + \gamma}{(2 \pi)^2} \mathcal{F}[\varphi](0)   \int_{\mathbb{S}^1} \mathcal{A}_{-2}(\sigma)  \, d\sigma  \bigg)
\end{align*} 
for any $\varphi \in \mathcal{S}(\R^d; \R^N)$.

With this definition a simple direct calculation also shows that
\[H^{\rm C}[G_0 e_k] = e_k \delta_0\]
and indeed $G^{\rm C}=G_0$.

\subsection{Alternative representation of \texorpdfstring{$G_n$}{Gn}}
We now connect the definitions made above to alternative representations considered by Morrey in \cite[chap.\ 6.2]{m66}. These alternative representations will provide us with the ability to directly deduce regularity, decay, and even homogeneity. Furthermore, this representation is promising for computational uses as only a finite surface integral in Fourier space is required for its evaluation.

We begin by setting $P = \lceil\frac{d+2n-1}{2}\rceil$, and $h= 2P+2-2n-d$, which we note satisfies $h\geq 1$. Then, we define
\begin{equation} \label{eq:morreykernel}
G_{n}^{\mathcal{M}}(\ell) = (-\Delta)^P \frac{c_{\rm vol}}{(2 \pi)^d}  \int_{\mathbb{S}^{d-1}} \mathcal{A}_{2n-2}(\sigma) J_{-1-h}(\ell \cdot \sigma) \,d\sigma,
\end{equation}
for $\ell \neq 0$, where $\Delta$ is the Laplacian, and
\begin{align*}
J_l(w) &= l! (-iw)^{-l-1}, &\text{ for } l \geq 0\\
J_l(w) &= \frac{1}{(-l-1)!} (iw)^{-l-1}\Big(-\log(-iw) + \sum_{j=1}^{-l-1} j^{-1}\Big), &\text{ for } l < 0.
\end{align*}
We remark that these functions satisfy $J_l'(w)=iJ_{l+1}(w)$. We also use the following definition from \cite[Def.\ 6.1.4]{m66}.
\begin{definition}
  A function $\varphi \colon \R^d \to \R^{N \times N}$ is called essentially homogeneous of order $s$ if either
  \begin{itemize}
  \item $s<0$ and the components of $\varphi$ are all positively homogeneous of order $s$, or 
  \item $s \in \N_0$ and $\varphi(\ell) = \varphi_1(\ell) + \varphi_2(\ell) \log \lvert \ell \rvert$, where the components of $\varphi_1$ are positively homogeneous of order $s$, and the components of $\varphi_2$ are homogeneous polynomials in the components of $\ell$ of degree $s$.
  \end{itemize}
\end{definition}
We directly want to point out though, that for the purpose of our results here we are only interested in the cases $s<0$ and $s=0$. So either $\varphi$ is positively homogeneous of a negative degree, or $\varphi(\ell) = \varphi_1(\ell) + A \log \lvert \ell \rvert$, where $A$ is a constant and $\varphi_1$ is $0$-homogeneous. With this definition in place, we state the following result.
\begin{lemma}
$G_{n}^{\mathcal{M}}$ is well-defined, smooth for $\ell \neq 0$, and $G_{n}^{\mathcal{M}}$ is essentially $(2-2n-d)$-homogeneous. Furthermore, $G_{n}^{\mathcal{M}}$ is positively $(2-2n-d)$-homogeneous if either $2-2n-d<0$ or $d$ is odd.
\end{lemma}
\begin{proof}
This is part of \cite[Thm. 6.2.1]{m66}. Note that $\mathcal{A}_{2n}$ being matrix-valued does not change the argument.
\end{proof}
In our setting, the only case that might fail to be positively homogeneous is $G_{0}^{\mathcal{M}}$ for $d=2$. As pointed out above, in that case essentially homogeneous means that $G_{0}^{\mathcal{M}} (\ell) = \varphi(\ell) + C_1 \log \lvert \ell \rvert$, where $\varphi$ is positively $0$-homogeneous and $C_1 \in \R^{N \times N}$.

We will now show that $G_{n}$ and $G_{n}^{\mathcal{M}}$ are in fact identical. Let us first give a short motivation for $G_{n}^{\mathcal{M}}$. The basic idea of the definition \eqref{eq:morreykernel} is to use the homogeneity of $\mathcal{A}_{2n-2}$ to isolate the radial component in the Fourier integral and then integrate it out. Naïvely, this approach leads to a integral of the form $\int_0^\infty r^l e^{irw} \,dr$ which does not exist for any $l$ and real $w$. The idea behind the $P$ Laplacians is to ensure that the resulting $l$ is non-negative. If one then adds a small imaginary part to $w$, the integral is indeed finite and can be computed. This leads to the $J_l$.

\begin{lemma}
On $\R^d \backslash \{0\}$ the distribution $G_{n}$ is represented by a function which we will also call $G_{n}$ and we have $G_{n}= G_{n}^{\mathcal{M}}$ on $\R^d \backslash \{0\}$.
\end{lemma}
\begin{proof}
As discussed in the motivation directly before the Lemma, we first let $\Im(w)>0$ and $l \geq 0$. Then
\[J_l(w) = l! (-iw)^{-l-1} = \int_0^{\infty} r^l e^{iwr}\,dr.\]
Take any $\delta>0$, and let $[i,w]$ be the line segment in the complex plane, for which we note that $z\in[i,w]$ implies $\Im(z)>0$.

Define
\[\varphi_0(z) := \sum_{k=0}^{\infty} \frac{z^k}{(k+1) (k+1)!},\]
which is easily seen to satisfy $(z\varphi_0(z))' = \frac{e^z-1}{z}$. Using the definition of $J_0$ and the definition and property of $\varphi_0$ just noted, we then find
\begin{align}
  J_{-1} (w)
  &= - \log(-iw) \nonumber\\
  &= i\int_{[i,w]} J_0(z)\,dz\nonumber\\
  &= i\int_{[i,w]} \int_{\delta}^\infty e^{izr}\,dr  + \frac{e^{iz\delta} -1}{iz} \,dz \nonumber\\
  &= \int_\delta^\infty r^{-1}\big( e^{iwr}- e^{-r}\big)\,dr +\delta iw \varphi_0(\delta i w)+ \delta \varphi_0(-\delta)\nonumber\\
&= \int_{\delta}^\infty r^{-1} e^{iwr}\,dr  + p^\delta_0(w) + \delta iw \varphi_0(\delta i w), \label{eq:J-1form}
\end{align}
where we define
\[p^\delta_0(w) :=   -\int_\delta^\infty r^{-1} e^{-r}\,dr + \delta \varphi_0(-\delta).\]
In \eqref{eq:J-1form}, the first term is the desired integral for $l=-1$, the second term is a renormalization of the singular part for $r$ close to $0$, and the third term is of lesser importance and will go to zero at the end.

We can simplify $p^\delta_0(w)$ further, as the exponential integral $E_1$ satisfies
\[E_1(z) =\int_{[z,\infty]} u^{-1} e^{-u}\,du = -\gamma - \log(z) + z \varphi_0(-z)\]
for $z \in \C \backslash \R_{\leq 0}$, where $\gamma$ is the Euler–Mascheroni constant, and $[z,\infty]$ is any contour in $\C \backslash \R_{\leq 0}$ connecting $z$ to positive real infinity. It follows that
\[p^\delta_0(w) = -\int_\delta^\infty r^{-1} e^{-r}\,dr + \delta \varphi_0(-\delta) = \gamma + \log(\delta).\]

Now, defining \[\varphi_h(z) = \sum_{k=0}^{\infty} \frac{z^k}{(k+1) (k+h+1)!}\]
and
\begin{align*}
p^\delta_h(w) &:= i\int_{[0,w]} p^\delta_{h-1}(z)\,dz - \frac{1}{h} \delta^{-h}\\
&= \frac{i^h}{h!}  w^h \big( \gamma + \log(\delta) \big)  - \sum_{k=0}^{h-1} \frac{i^k}{k!}  w^k \frac{1}{h-k} \delta^{k-h},
\end{align*}
for $h \geq 0$ we may calculate inductively that
\begin{align*}
J_{-h-1} (w) &= i\int_{[0,w]} J_{-h}(z)\,dz\\
&= i\int_{[0,w]} \int_{\delta}^\infty r^{-h}e^{izr}\,dr  + p^\delta_{h-1}(z) + \delta (iz)^h \varphi_{h-1}(\delta i z) \,dz \\
&= \int_{\delta}^\infty r^{-h-1} e^{iwr}\,dr  + p^\delta_h(w) + \delta (iw)^{h+1} \varphi_{h}(\delta i w).
\end{align*}

This calculation was restricted to $\Im(w)>0$ to ensure that the integrals converge. But for $h \geq 1$ that is true even for real $w$. So we now can take the limit $\Im(w) \to 0$ and see that for all $w \in \R$, $\delta >0$, and $h \geq 1$, we have that
\begin{align*}
J_{-h-1} (w) &= \int_{\delta}^\infty r^{-h-1} e^{iwr}\,dr  + p^\delta_h(w) + \delta (iw)^{h+1} \varphi_{h}(\delta i w).
\end{align*}

Using the definition in \eqref{eq:morreykernel}, we have obtained the following representation for $G_{n}^{\mathcal{M}}$:
\begin{align*}
  G_{n}^{\mathcal{M}}(\ell)
  =  \frac{c_{\rm vol }(-\Delta)^P}{(2 \pi)^d} \bigg( &\int_{\R^d \backslash B_{\delta}(0)} \frac{\mathcal{A}_{2n-2}(k)}{ \lvert k \rvert^{2P}} e^{i\ell k}\,dk \\
  &\quad+ \int_{\mathbb{S}^{d-1}} \mathcal{A}_{2n-2}(\sigma) p^\delta_h(\ell \sigma)d\sigma\\
  &\qquad+\int_{\mathbb{S}^{d-1}}\mathcal{A}_{2n-2}(\sigma)\delta (i\ell \sigma)^{h+1} \varphi_{h}(\delta i \ell \sigma) \,d\sigma\bigg).
\end{align*}
We now inspect the three integrals in turn. Clearly, for the first 
\[(-\Delta)^P  \int_{\R^d \backslash B_{\delta}(0)} \frac{\mathcal{A}_{2n-2}(k)}{ \lvert k \rvert^{2P}} e^{i\ell k}\,dk = \int_{\R^d \backslash B_{\delta}(0)} \mathcal{A}_{2n-2}(k)e^{i\ell k}\,dk\]
in the sense of tempered distributions. 
For the second integral, as $p^\delta_h$ is a polynomial of degree $h = 2P +2 -2n -d$, we find $ (-\Delta)^P p^\delta_h ( \ell \sigma)=0$ except when $d=2$ and $n=0$, in which case  $ (-\Delta)^P p^\delta_h ( \ell \sigma)= \gamma + \log \delta$. The integrand in the final integral contains the factor
$(-\Delta)^P \delta (i\ell \sigma)^{h+1} \varphi_{h}(\delta i \ell \sigma)$, which goes to $0$ uniformly on compact sets as $\delta \to 0$. Since $\delta$ was arbitrary, we may pass to this limit, and we do indeed find $G_m= G_m^{\mathcal{M}}$ on $\R^d \backslash \{0\}$.
\end{proof}

%
%

After having constructed the kernels $G_n$ and having established their properties, we can now prove the expansion of the lattice Green's function.

\begin{proof}[Proof of Theorem \ref{thm:greensfunctionexpansion}]
We have already established the regularity and homogeneity properties of the $G_n$. In the following, fix $j \geq 0$ and $\boldsymbol{\rho}=(\rho_1, \dots, \rho_j)  \in \mathcal{R}^j$. We then still need to show that 
\begin{equation} \label{eq:GGpdiffstrong}
\bigg\lvert D_{\boldsymbol{\rho}} \G(\ell) - \sum_{i=0}^p D_{\boldsymbol{\rho}} G_n(\ell)\bigg\rvert \leq C_{p} \abs{\ell}^{-2p-d-j}
\end{equation}
for all $\ell$, $p$. We claim it is in fact sufficient to show the result holds for any $p$ but where the error decays with two orders less, i.e.
\begin{equation} \label{eq:GGpdiffweak}
\bigg\lvert D_{\boldsymbol{\rho}} \G(\ell) - \sum_{i=0}^p D_{\boldsymbol{\rho}} G_n(\ell)\bigg\rvert \leq C_{p} \abs{\ell}^{2-2p-d-j}
\end{equation}
for all $\ell$. If this estimate holds, then using equation \eqref{eq:GGpdiffweak} with $p+1$ and the triangle inequality entails
\begin{equation*}
\bigg\lvert D_{\boldsymbol{\rho}} \G(\ell) - \sum_{i=0}^p D_{\boldsymbol{\rho}} G_n(\ell) \bigg\rvert \leq \lvert D_{\boldsymbol{\rho}} G_{p+1}(\ell)\rvert + C_{p+1} \abs{\ell}^{-2p-d-j},
\end{equation*}
but since we have that $D_{\boldsymbol{\rho}} G_{p+1}(\ell) = O(\lvert \ell \rvert^{-2p-d-j})$ due to the homogeneity of $G_{p+1}$, this establishes \eqref{eq:GGpdiffstrong}.

To show \eqref{eq:GGpdiffweak}, fix a smooth cutoff $\eta \colon \R^d \to [0,1]$ with $\eta(k)=1$ in a neighbourhood of $0$ and $\supp \eta \subset \subset \mathcal{B}$. Then set $G_{n}^\eta := c_{\rm vol}\mathcal{F}^{-1}_c [\mathcal{A}_{2n-2} \eta]$ and $G_{n}^{1-\eta} := c_{\rm vol}\mathcal{F}^{-1}_c [\mathcal{A}_{2n-2} (1-\eta)]$. If $n=0$ and $d=2$, set
\[ G_0^\eta :=  \lim_{\delta \to 0} \Big( \frac{c_{\rm vol}}{(2 \pi)^2}\int_{\R^2 \backslash B_{\delta}(0)} \mathcal{A}_{-2}(k) \eta(k) e^{ik \ell} \, dk + c_{\rm vol} \frac{\log \delta + \gamma}{ (2 \pi)^2} \int_{\mathbb{S}^1} \mathcal{A}_{-2}(\sigma)  \, d\sigma  \Big) \]
instead. If $2n-2 > -d $, note that $\mathcal{A}_{2n} \eta$ is integrable and has support in $\mathcal{B}$, therefore
\[ G_{n}^\eta =\frac{1}{\lvert \mathcal{B} \rvert}\int_{\mathcal{B}} \mathcal{A}_{2n-2}(k) \eta(k) e^{ik \ell} \, dk, \]
as well as
\[G_{0}^\eta =\lim_{\delta \to 0} \Big( \frac{1}{\lvert \mathcal{B} \rvert}\int_{\mathcal{B} \backslash B_{\delta}(0)} \mathcal{A}_{-2}(k) \eta(k) e^{ik \ell} \, dk + (\log \delta + \gamma) \frac{1}{\lvert \mathcal{B} \rvert} \int_{\mathbb{S}^1} \mathcal{A}_{-2}(\sigma)  \, d\sigma  \Big) \]

In all cases we find
\begin{align}
& \big(D_{\boldsymbol{\rho}} \mathcal{G}(\ell) - \sum_{n=0}^p D_{\boldsymbol{\rho}} G^{\eta}_{n}(\ell)\big) \lvert \ell \rvert^{2m} \nonumber \\
&\quad = \frac{1}{\lvert \mathcal{B} \rvert}\int_{\mathcal{B}} (-\Delta)^m \Big(\big( \hat{H}^{-1}(k)- \eta(k) \sum_{n=0}^p \mathcal{A}_{2n-2}(k) \big)    \prod_{s=1}^j (e^{i\rho_s k} -1)\Big) e^{ik \ell}  \, dk\nonumber  \\
&\quad = \frac{1}{\lvert \mathcal{B} \rvert}\int_{\mathcal{B}} r_m(k) e^{ik \ell}  \, dk, \label{eq:Gseries1}
\end{align}
with
\[
r_m(k) := (-\Delta)^m \Big(\big( \hat{H}^{-1}(k)- \eta(k) \sum_{n=0}^p \mathcal{A}_{2n-2}(k) \big)    \prod_{s=1}^j (e^{i\rho_s k} -1)\Big)
\]
as long as $m$ is small enough to ensure that $r_m$ is integrable. The smoothness of $\eta$, and the properties of $\hat{H}^{-1}$ and $\mathcal{A}_{2n}$ discussed in \S\ref{sec:four-mult-seri} ensure that $r_m$ is smooth away from $k=0$ and is bounded by $C \lvert k \rvert^{2p-2m + j}$ in a neighbourhood of $k=0$. It follows that $r_m \in L^1$ if $2m \leq 2p  + j + (d-1)$ and therefore \eqref{eq:Gseries1} implies
\[D_{\boldsymbol{\rho}} \mathcal{G} - \sum_{n=0}^p D_{\boldsymbol{\rho}} G^{\eta}_{n} = O(\lvert \ell \rvert^{2-2p-d-j}). \]
For $G_{n}^{1-\eta}$, we observe that $D^{\alpha} \big( \mathcal{A}_{2n-2} (1-\eta)\big)$ is integrable for all sufficiently large $\lvert \alpha \rvert$, and therefore $G_{n}^{1-\eta}$ has super-algebraic decay at infinity. The same therefore also holds true for any $D_{\boldsymbol{\rho}} G_{n}^{1-\eta}$.

The only remaining claim now is the uniqueness. Assume we have two such series of kernels $(G_n)_n$ and $(\tilde{G}_n)_n$. By induction, for $p\in \N_0$ let us assume we already know that $G_n = \tilde{G}_n$ for all $n<p$, we need to show that $G_p = \tilde{G}_p$.

First exclude the case where both $p=0$ and $d=2$. Then we know that $G_p$ and $\tilde{G}_p$ are positively homogeneous of order $(2-2p-d)$. Using both estimates of order $p$ without any derivatives, we obtain
\begin{align*}
\Big\lvert  G_p(\ell) -  \tilde{G}_p(\ell)\Big\rvert &\leq \Big\lvert  \G(\ell) - \sum_{n=0}^p  G_n(\ell)\Big\rvert + \Big\lvert  \G(\ell) - \sum_{n=0}^p  \tilde{G}_n(\ell)\Big\rvert \\
 &\leq (C_{0,p} + \tilde{C}_{0,p}) \abs{\ell}^{-2p-d}
\end{align*}
for all $\ell \in \Lambda$. Combined with the homogeneity, we thus have
\[\Big\lvert  G_p\Big(\frac{\ell}{\lvert \ell \rvert}\Big) -  \tilde{G}_p\Big(\frac{\ell}{\lvert \ell \rvert}\Big)\Big\rvert \leq (C_{0,p} + \tilde{C}_{0,p}) \abs{\ell}^{-2}.\]
Given any $\eps >0$ we can therefore find an $R>0$ such that
\begin{equation} \label{eq:uniquenessapprox}
\lvert  G_p(x) -  \tilde{G}_p(x)\rvert \leq \eps
\end{equation}
for all $x \in \{ \frac{\ell}{\lvert \ell \rvert} \colon \ell \in \Lambda, \lvert \ell \rvert > R \}$. As this set is dense in $\partial B_1(0)$ and both functions are continuous, \eqref{eq:uniquenessapprox} holds for all $x \in \partial B_1(0)$. As $\eps >0$ was arbitrary we thus have $G_p = \tilde{G}_p$ on $\partial B_1(0)$ and by homogeneity on all of $\R^d \backslash \{0\}$.

In the case that $p=0$ and $d=2$ the argument is only slightly more complex. We write $G_0 = A \log \lvert \ell \rvert + \varphi$ and $\tilde{G}_0 = \tilde{A} \log \lvert \ell \rvert + \tilde{\varphi}$ with $A,\tilde{A} \in \R^{N \times N}$ and $\varphi, \tilde{\varphi}$ both $0$-homogeneous. Then
\[\Big\lvert \tilde{A} \log \lvert \ell \rvert + \tilde{\varphi}(\ell) -   A \log \lvert \ell \rvert - \varphi(\ell)\Big\rvert \leq (C_{0,0} + \tilde{C}_{0,0}) \abs{\ell}^{-2}\]
 for all $\ell \in \Lambda$. As $\varphi$ and $\tilde{\varphi}$ are bounded, we can divide by $\log \lvert \ell \rvert$ and send $\lvert \ell \rvert \to \infty$ to see that $A=\tilde{A}$. The identity $\varphi=\tilde{\varphi}$ then follows the same way as before.
\end{proof}

We can now rewrite the truncated multipole expansion of Theorem \ref{thm:structurewithmoments} in continuum terms.
\begin{lemma} \label{lem:MPCMP}
\begin{align*}
\sum_{i=0}^{p-1} \sum_{k=1}^N b^{(i,k)} : D_{\rm \mathcal{S}}^i \mathcal{G}_k = \sum_{k=1}^N \sum_{n=0}^{\big\lfloor \frac{p-1}{2} \big\rfloor} \sum_{i=0}^{p-1-2n} a^{(i,n,k)} : \nabla^i (G_{n})_{\cdot k} + w^{\rm MP},
\end{align*}
where
\[ \lvert D^j w^{\rm MP} \rvert \leq C_{j} \abs{\ell}^{2-d-j-p}\]
for all $j \in \N_0$.
\end{lemma}
\begin{proof}
This is an immediate consequence of the expansion of the lattice Green's function from Theorem \ref{thm:greensfunctionexpansion} and a Taylor expansion of the discrete differences in $D_{\boldsymbol{\rho}}G_{n}$.
\end{proof}

\noindent As a consequence we obtain Theorem~\ref{thm:structurewithmomentscont}.

\begin{proof}[Proof of Theorem~\ref{thm:structurewithmomentscont}]
  The result follows by combining Theorem \ref{thm:structurewithmoments} and Lemma \ref{lem:MPCMP}.
\end{proof}

\section{Proofs - Far Field Expansions for Crystalline Defects} \label{sec:prooffarfieldexpansion}

In this section we prove the results of \S~\ref{sec:applications}. Before we get to the main results, let us start with a preliminary proof.
\begin{proof}[Proof of Proposition~\ref{prop:uCLE}]
$\delta \mathcal{E}(u)(\ell)$ is summable according to \cite[Lemma 10]{EOS2016}. The same then holds true for $H[u](\ell)$. The idea of the proof is to use a cutoff $\eta_R$ and sum by parts so that only the far away annulus $B_{2R} \setminus B_R$ contributes. Here, the linearisation and discretisation errors are small and we can go from the nonlinear forces to the linearised forces and all the way to the continuum forces. So, consider a smooth cutoff function $\eta_R$, such that $\eta_R \colon \R^d \to [0,1]$ satisfies $\eta_R(z)=1$ for $\lvert z \rvert \leq R$ and $\eta_R(z)=0$ for $\lvert z \rvert \geq 2R$, as well as $\lvert \nabla \eta_R \rvert \lesssim R^{-1}$. We then have
\begin{align*}
\sum_\ell \delta \mathcal{E}(u)(\ell)_i &= \lim_{R \to \infty} \sum_\ell \eta_R(\ell) \delta \mathcal{E}(u)(\ell)_i \\
 &= \lim_{R \to \infty} \sum_\ell \nabla V(Du)[e_i \otimes D\eta_R]\\
 &= \lim_{R \to \infty} \sum_\ell \big(\nabla V(Du)-\nabla V(0)\big)[e_i \otimes D\eta_R]\\
 &= \lim_{R \to \infty} \sum_\ell \big(\nabla^2 V(0)\big)[e_i \otimes D\eta_R, Du]\\
 &= \sum_\ell H[u](\ell)_i,
\end{align*}
as the linearisation error sums to $O(R^{-1})$. Furthermore,
\begin{align*}
\lim_{R \to \infty} &\sum_\ell \big(\nabla^2 V(0)\big)[e_i \otimes D\eta_R, Du]\\
 &= \lim_{R \to \infty} \int \mathbb{C}[e_i \otimes \nabla \eta_R, \nabla u]\,d\ell\\
 &= \lim_{R \to \infty} \int_{\R^2 \setminus B_1(\hat{x})} -\divo \big(\mathbb{C}[\nabla u] \big)_i \eta_R \,d\ell -\int_{\partial B_1(\hat{x})}\mathbb{C}[\nabla u] \nu \,d\sigma \\
  &= -\int_{\partial B_1(\hat{x})}\mathbb{C}[\nabla u] \nu \,d\sigma,
\end{align*}
as the discretisation error is also $O(R^{-1})$.
\end{proof}

Now let us come to the main topic of this section, the proofs of Theorem~\ref{thm:pointdef} and Theorem~\ref{thm:screw}.

At the outset, we recall from the setting discussed in \S\ref{sec:genresults} and \S\ref{sec:applications} that we assume the site potentials (and hence the potential energy functional) are of class $C^K$, the number of derivatives we want to estimate is $J \geq 2$ and the dimension of the problem is $d$. We now prove the expansion and the error estimates by induction over $p \geq 0$ as long as
\begin{equation}\label{eq:KJinequality}
\max \Big\{0,\Big\lfloor \frac{p-1}{d} \Big\rfloor\Big\}\leq K-J-2 
\end{equation}
in the point defect case and
\begin{equation*}
p \leq K-J-2
\end{equation*}
in the case of the screw dislocation, where we always have $d=2$. We also note that all results only need to be proven for $\lvert \ell \rvert \geq R$ for some sufficiently large $R$.

\subsection{The case \texorpdfstring{$p=0$}{p=0}} The case $p=0$ is a consequence of results in \cite{EOS2016} for both the point defect case and the case of screw dislocation. 
To be precise, it follows from the assumptions of both Theorem~\ref{thm:pointdef} and Theorem~\ref{thm:screw} that $J \leq K-2$. In this setting, we may apply Theorem~1 in \cite{EOS2016}, which states that the solution $\bar{u}$ in the point defect cases satisfies
\begin{equation} \label{eq:PDerrorlowestorder}
\lvert D^j \bar{u} (\ell) \rvert \lesssim \abs{\ell}^{1-d-j},\quad 1\leq j \leq J.
\end{equation}
In the screw dislocation case, we may similarly apply Theorem~5 in \cite{EOS2016}, implying that
\begin{equation}\label{eq:Dislerrorlowestorder}
  \lvert D^j (\bar{u}-u_{\rm CLE}) (\ell) \rvert \lesssim \abs{\ell}^{-1-j} \log \abs{\ell},\quad 1\leq j \leq J.
\end{equation}
In order to combine and summarise these results, we define
\[
  u_0 := \begin{cases}
    0 & \text{for the point defect,}\\
    u_{\rm CLE} & \text{for the dislocation,}
    \end{cases}
\]
so that \eqref{eq:PDerrorlowestorder} and \eqref{eq:Dislerrorlowestorder} imply
\begin{equation}\label{eq:roneestimate}
\lvert D^j (\bar{u}-u_0) (\ell) \rvert \lesssim \abs{\ell}^{1-d-j}\log \abs{\ell},\quad 1\leq j \leq J,
\end{equation}
while $\nabla u_0 \in C^\infty(\R^d\backslash\{0\};\R^N)$ with
\begin{equation}\label{eq:uzeroestimate}
\lvert \nabla^j u_0 (\ell) \rvert \lesssim \abs{\ell}^{-j},\quad \text{for all}\ j\geq 1.
\end{equation}

\subsection{The case \texorpdfstring{$p>0$}{p>0}}
We now proceed inductively from the case $p=0$.

\paragraph{Taylor expansion.} We follow the construction that we formally motivated in \S~\ref{sec:contdev} and expand the solution $\bar{u}$ by performing a Taylor expansion of the equilibrium equation for the forces
\[\delta \mathcal{E}(\bar{u})(\ell)=-\Div g(\ell)\]
for all $\ell$ sufficiently large, where $g=0$ in the screw dislocations case and $g$ has compact support for point defects.

Recall from \S\ref{sec:applications} that $\mathcal{E}$ inherits the same regularity as the homogeneous site potential $V$, and is therefore $C^K$. For such a variation of the energy we use the pointwise notation for the corresponding $\ell^2$ representative. Specifically, that means
\[\delta^{k+1} \mathcal{E}(u)[v_1, \dots, v_k](\ell) = - \Div \Big( \nabla^{k+1} V (Du(\ell))[Dv_1(\ell), \dots, Dv_k(\ell)]\Big).\]

 A Taylor expansion around the homogeneous lattice $u=0$ up to order $\tilde{K}\leq K-2$ gives
\begin{equation} \label{eq:expansionstep1}
\begin{aligned}
0 &=\Div g +\delta \mathcal{E}(\bar{u})(\ell) \\
&= \Div g+\sum_{k=0}^{\tilde{K}} \frac{1}{k!} \delta^{k+1} \mathcal{E}(0)[\bar{u}]^k(\ell) + \int_0^1 \frac{(1-t)^{\tilde{K}}}{\tilde{K}!} \delta^{\tilde{K}+2} \mathcal{E}(t\bar{u})[\bar{u}]^{(\tilde{K}+1)}(\ell)\,dt.
\end{aligned}
\end{equation}
We will determine the precise order $\tilde{K}$ for this expansion as part of our arguments later.
We note that we already have explicit expressions for the first two terms in the sum: A homogeneous (Bravais) lattice is always in equilibrium, $\delta \mathcal{E}(0)=0$, and the second term is  $\delta^2 \mathcal{E}(0)[\bar{u}] = H[\bar{u}]$. For convenience, we will define $\mathcal{T}_{p+1}^{(1)}$ to be the integral error term and also include the compactly supported $\Div g$, i.e.
\[\mathcal{T}_{p+1}^{(1)} := -\Div g - \int_0^1 \frac{(1-t)^{\tilde{K}}}{\tilde{K}!} \delta^{\tilde{K}+2} \mathcal{E}(t\bar{u})[\bar{u}]^{(\tilde{K}+1)}\,dt,\]
so that we may rewrite \eqref{eq:expansionstep1} as
\begin{equation}\label{eq:expansionstep1a}
  H[\bar{u}] = -\sum_{k=2}^{\tilde{K}} \frac{1}{k!} \delta^{k+1} \mathcal{E}(0)[\bar{u}]^k +\mathcal{T}_{p+1}^{(1)}.
\end{equation}

\paragraph{Expansion of solution.} Next, we make the inductive assumption that we have already obtained $u_i$ for $0 \leq i \leq p-1$, and that we can further decompose
\begin{equation}\label{eq:uiSplitting}
  u_i = u_i^{\rm C} + u_i^{\rm MP},
\end{equation}
which are respectively continuum- and lattice-valued functions, so that we may write
\begin{equation}\label{eq:sumExpansion}
  \bar{u} = \sum_{i=0}^{p-1} u_i + r_p = \sum_{i=0}^{p-1} \Big(u^{\rm C}_i+ u^{\rm MP}_i\Big)+ r_p,
\end{equation}
where $r_p$ is some remainder. These functions will be assumed inductively to satisfy certain properties which we now make precise:
\begin{itemize}
\item We will assume that the former terms, $u_i^{\rm C}\in C^{\infty}(\R^d\backslash B_{R_0}(0))$, $i \leq p-1$, have been constructed to solve a sequence of linear elliptic PDEs,
  \begin{equation}\label{eq:uCiPDE}
    H^{\rm C}[u_i^{\rm C}]= \mathcal{S}_i(u_0^{\rm C}, u_0^{\rm CMP}, \dots, u_{i-1}^{\rm C}, u_{i-1}^{\rm CMP})
  \end{equation}
  and satisfy decay estimates
  \begin{equation}\label{eq:uCestimate}
    \lvert \nabla^j u_i^{\rm C} (\ell)\rvert \leq C_j \lvert \ell \rvert^{2-d-j-i} \log^{i} \lvert \ell \rvert \quad\text{for all }j\in\N_0.
  \end{equation}
  We note that $\mathcal{S}_i$ are forcing terms which we determine as part of our argument.
\item We assume that the latter terms in \eqref{eq:uiSplitting} have been constructed to take the form
\begin{equation}\label{eq:uiMP_Form}
  u_i^{\rm MP} = \sum_{k=1}^N b^{(i,k)} : D_{\rm \mathcal{S}}^i \mathcal{G}_k
\end{equation}
for some constant tensors $b^{(i,k)}$. In particular, $H[u_i^{\rm MP}](\ell) =0$ for $|\ell|$ large enough, and we assume $u_0^{\rm MP}=0$.
\item We assume that the remainder estimate 
\begin{equation} \label{eq:errorprevorder}
  \lvert D^jr_{p}(\ell) \rvert \lesssim \abs{\ell}^{2-d-p-j}  \log^{p} \abs{\ell}
\end{equation}
holds for $j=1,\dots,J$.

\end{itemize}
In light of \eqref{eq:uiMP_Form}, whenever useful, we may also rewrite the multipole contributions in terms of smooth continuum kernels by applying Lemma \ref{lem:MPCMP}. Indeed, by combining the terms of equal homogeneity, Lemma~\ref{lem:MPCMP} implies that we may write
\[\sum_{i=1}^{p-1} u_i^{\rm MP} =  \sum_{i=1}^{p-1} u_i^{\rm CMP} + w^{\rm MP}_p, \]
where $u_i^{\rm CMP} \in C^\infty(\R^d\backslash \{0\};\R^N)$ is positively homogeneous of degree $(2-d-i)$ which implies
\begin{equation}\label{eq:uCMPestimate}
  \lvert \nabla^j u_i^{\rm CMP} \rvert \leq C_{j} \lvert\ell\rvert^{2-d-j-i}
\end{equation}
for all $j \in \N_0$, and we have the additional error estimate
\begin{equation} \label{eq:wpestimate}
\lvert D^j w^{\rm MP}_p \rvert \leq C_{j} \abs{\ell}^{2-d-j-p}
\end{equation} 
for all $j \in \N_0$. Using the functions $u^{\rm CMP}$, we therefore have the representation
\begin{equation}\label{eq:sumExpansion2}
  \bar{u} = \sum_{i=0}^{p-1} \Big(u^{\rm C}_i+u^{\rm CMP}_i\Big)+w^{\rm MP}_p+r_p.
\end{equation}

Using the inductive assumptions made above, we may substitute \eqref{eq:sumExpansion} into the left-hand side of \eqref{eq:expansionstep1a}, and write
\begin{equation} \label{eq:expansionstep2}
\sum_{i=0}^{p-1} H[u_i^{\rm C}] + H[r_p] = -\sum_{k=2}^{\tilde{K}} \frac{1}{k!} \delta^{k+1} \mathcal{E}(0)[\bar{u}]^k +\mathcal{T}_{p+1}^{(1)}.
\end{equation}
The goal is now to split the remainder $r_p$ into the next continuum contribution $u_p^{\rm C}$ and a new intermediary term $s_p$, and find the correct equation for $u_p^{\rm C}$ so the residual becomes even smaller in the far-field. Then we can apply Theorem \ref{thm:structurewithmoments} to $s_p$ to split this term yet further into multipole terms up to order $p$ and a new remainder $r_{p+1}$ with the desired improved decay.

\paragraph{Taylor expansion error.}
With the assumptions above, we turn back to the Taylor expansion \eqref{eq:expansionstep1}, and establish the choice of order, $\tilde{K}$. Our aim is to fully resolve all orders of decay up to and including $\abs{\ell}^{-d-p}$, and we have the decay estimates $\lvert D^j \bar{u} \rvert \lesssim \abs{\ell}^{2-d-j}$ and the stronger $\lvert D^j \bar{u} \rvert \lesssim \abs{\ell}^{1-d-j}$ in the case of a point defect, as $u_0=0$. Formally, this means that we expect that $\lvert \delta^{k+1} \mathcal{E}(0)[\bar{u}^{\otimes k}] (\ell) \rvert \lesssim \abs{\ell}^{(1-d)k-1}$ or $\abs{\ell}^{-dk-1}$ in the case of a point defect. These orders of decay come from estimating a $k$-fold tensor product of $D\bar{u}$ with the decay mentioned, and one further additional power arising from the discrete divergence. In particular, to ensure that the order of decay matches or exceeds that of $\lvert \ell\rvert_0^{-d-p}$, we require $\tilde{K} \geq \frac{p}{d-1}+1$, or in the case of a point defect $\tilde{K} \geq \frac{p-1}{d}+1$. Indeed, this is satisfied by assumption on $K$, if we set $\tilde{K} := K-J-1$. We also note that the expression of $\mathcal{T}_{p+1}^{(1)}$ involves $\tilde{K}+2=K-J+1$ derivatives of $\mathcal{E}$ and thus $V$, as well as one divergence, so $J-2$ further derivatives are available. Let us use these considerations and the choice $\tilde{K} := K-J-1$ to rigorously estimate $\mathcal{T}_{p+1}^{(1)}$.

\begin{lemma} \label{lem:firsterror}
With $\tilde{K}=K-J-1$, the remainder term in the Taylor expansion \eqref{eq:expansionstep1} satisfies
\[\big\lvert D^j \mathcal{T}_{p+1}^{(1)} (\ell)\big\rvert \lesssim \abs{\ell}^{-d-p-1-j} \]
for all $j=0, \dots, J-2$.
\end{lemma}
\begin{proof}
For $\ell$ large enough, $-\Div g(\ell) =0$, hence we have
\begin{align*}
\big\lvert D^j \mathcal{T}_{p+1}^{(1)} \big\rvert &= \bigg\lvert D^j \int_0^1 \frac{(1-t)^{\tilde{K}}}{\tilde{K}!} \delta^{\tilde{K}+2} \mathcal{E}(t\bar{u})[\bar{u}]^{\tilde{K}+1}\,dt \bigg\rvert \\
&\lesssim \max_{t \in [0,1]} \Big\lvert D^{j+1} \Big( \nabla^{\tilde{K}+2}V(tD\bar{u})\big[D\bar{u}\big]^{\tilde{K}+1} \Big)  \Big\rvert.
\end{align*}
In the case of a screw dislocation we can use the estimates \eqref{eq:roneestimate} and \eqref{eq:uzeroestimate} for $u_0$ and $r_1$ to see that $\lvert D^j \bar{u} \rvert \lesssim \abs{\ell}^{-j}$ for $1 \leq j \leq J$.  By distributing the outer discrete derivatives over the tensor product and using this estimate, we deduce that
\[\lvert D^j \mathcal{T}_{p+1}^{(1)}(\ell) \rvert \lesssim \abs{\ell}^{-\tilde{K}-j-2}.\]
This proves the result in this case as $\tilde{K}=K-J-1 \geq p+1$ and $d=2$.

For a point defect, we have $u_0=0$. More specifically, according to \eqref{eq:PDerrorlowestorder} we have $\lvert D^j \bar{u} \rvert \lesssim \abs{\ell}^{1-d-j}$ for $1 \leq j \leq J$. Using this estimate and distributing the outer discrete derivative as in the previous case, it follows that
\[\lvert D^j \mathcal{T}_{p+1}^{(1)}(\ell) \rvert \lesssim \abs{\ell}^{-d(\tilde{K}+1)-j-1}.\]
And since $\tilde{K} = K-J-1 \geq \lfloor \frac{p-1}{d}\rfloor+1$ we can estimate the exponent by
\begin{align*}
d(\tilde{K}+1)+j+1 &\geq d \Big(\Big\lfloor \frac{p-1}{d}\Big\rfloor +2\Big) +j+1 \\
&\geq d \Big(\frac{p}{d}-1\Big) +2d +j+1 \\
&= p +d + j +1.\qedhere
\end{align*}
\end{proof}

\paragraph{Error terms due to discrete far-field residuals.}
With $\mathcal{T}_{p+1}^{(1)}$ controlled, we now go back to work on \eqref{eq:expansionstep2}. As the next step we want to replace $\bar{u}$ on the right hand side by a smooth continuum approximation which naturally gives us another error term from the remaining discrete far-field terms. Recalling \eqref{eq:sumExpansion2},
we have
\begin{equation*}
  \bar{u} = \sum_{i=0}^{p-1}\Big(u^{\rm C}_i+u^{\rm CMP}_i\Big) + w^{\rm MP}_p+r_p.
\end{equation*}
We choose to decompose $r_p = u^{\rm C}_p+s_p$, where we will specify $u^{\rm C}_p$ as the solution to a PDE later in the construction. For now, take any function $u^{\rm C}_p \in C^\infty(\R^d \backslash B_{R}(0);\R^N)$ that satisfies
\begin{equation} \label{eq:upcdecay}
  \lvert \nabla^j u_p^{\rm C} \rvert \leq C_j \abs{\ell}^{2-d-p-j} \log^{p} \abs{\ell}
\end{equation}
for all $j \in \N$, some radius $R>0$, and some constants $C_j>0$.

Then define $\bar{u}_p := \sum_{i=0}^p u^{\rm C}_i +\sum_{i=0}^{p-1}u^{\rm CMP}_i$, so that
\begin{equation}\label{eq:sumExpansion3}
\bar{u}= \bar{u}_p  + s_p + w_p^{\rm MP}.
\end{equation}
Based on that, we can substitute \eqref{eq:sumExpansion3} into the series on the right-hand side of \eqref{eq:expansionstep2}, giving
\begin{equation} \label{eq:expansionstep3}
\sum_{i=0}^p H[u_i^{\rm C}] + H[s_p] = -\sum_{k=2}^{\tilde{K}} \frac{1}{k!} \delta^{k+1} \mathcal{E}(0)[\bar{u}_p]^k +\mathcal{T}_{p+1}^{(1)}+ \mathcal{T}_{p+1}^{(2)},
\end{equation}
where
\begin{align*}
  \mathcal{T}_{p+1}^{(2)}&= \mathcal{T}_{p+1}^{(2)}(u_0^{\rm C}, u_0^{\rm CMP}, \dots, u_{p-1}^{\rm C}, u_{p-1}^{\rm CMP}, u_{p}^{\rm C}, s_p, w_p^{\rm MP})\\
  &:= -\sum_{k=2}^{\tilde{K}} \frac{1}{k!} \Big(\delta^{k+1} \mathcal{E}(0)[\bar{u}]^k-\delta^{k+1} \mathcal{E}(0)[\bar{u}_p]^k\Big).
\end{align*}
collects all the terms that contain contribution from $s_p$ or $w_p^{\rm MP}$. To estimate derivatives of $w_p^{\rm MP}$, we can use \eqref{eq:wpestimate}. For $s_p$ we can combine \eqref{eq:errorprevorder} and \eqref{eq:upcdecay} to obtain
\begin{equation} \label{eq:errorprevordervonly}
  \lvert D^j s_p \rvert \lesssim \abs{\ell}^{2-d-p-j}  \log^{p} \abs{\ell}.
\end{equation}
for $j=1,\dots,J$.

This allows us to estimate the second error term.
\begin{lemma} \label{lem:seconderror}
The following estimate holds:
\begin{align*}
\lvert D^j \mathcal{T}_{p+1}^{(2)} \rvert &\lesssim \abs{\ell}^{1-2d-p-j} \log^{p} \abs{\ell}\\
&\lesssim \abs{\ell}^{-d-1-p-j} \log^{p} \abs{\ell}.
\end{align*}
for all $j=0, \dots, J-2$.
\end{lemma}
\begin{proof}
  We note that $\mathcal{T}_{p+1}^{(2)}$ can be written as a sum of differences,
  \[
    \mathcal{T}_{p+1}^{(2)}  = -\sum_{k=2}^{\tilde{K}} \frac{1}{k!} \Big(\delta^{k+1} \mathcal{E}(0)[\bar{u}]^k-\delta^{k+1} \mathcal{E}(0)[\bar{u}_p]^k\Big).
  \]
For the $k=2$ term, we use \eqref{eq:sumExpansion3} and apply the discrete product rule to obtain
\begin{align*}
\Big \lvert D^j \delta^{3} \mathcal{E}(0)&[\bar{u}_p + s_p + w_p^{\rm MP}]^2 - D^j\delta^{3} \mathcal{E}(0)[\bar{u}_p]^2 \Big \rvert \\
&\lesssim  \sum_{i=0}^{j+1} \lvert D^{i+1} (s_p + w_p^{\rm MP}) \rvert \sup_{\rho \in B_R \cap \La}\lvert D^{j+2-i} (-2\bar{u}_p + s_p + w_p^{\rm MP})(\cdot+\rho) \rvert \\
&\lesssim  \sum_{i=0}^{j+1} \lvert D^{i+1} (s_p + w_p^{\rm MP}) \rvert  \, \abs{\ell}^{-d-j+i}.
\end{align*}
We note that the shifts $\rho$ appearing in the first estimate which arise from the discrete product rule remain below some finite radius $R=R(J)$. Using the estimates for $w_p^{\rm MP}$ and $s_p$ gives the desired overall estimate.

For $k \geq 3$ a similar argument provides the same estimate, as any additional factor just improves the estimate by at least $\abs{\ell}^{-1} \log \abs{\ell} \lesssim 1$.
\end{proof}

\paragraph{Continuum approximation.}
We want to construct a sequence of PDEs to resolve \eqref{eq:expansionstep3} to sufficient accuracy and thus describe the far-field behaviour of $\bar{u}$. To that end, we note that a finite difference of any smooth function $v$ can be Taylor expanded to give
\[
  D_\rho v(\ell) = \sum_{k=1}^{M-1}\frac1{k!}\nabla^k v(\ell)[\rho]^k+\frac{1}{M!}\nabla^{M}v(\ell+\theta\rho)[\rho]^M,
\]
for some $\theta\in[0,1]$. By iterating this process of Taylor expansion, a similar series approximation can be constructed for higher-order finite differences. In particular, for a second-order finite difference, we obtain
\begin{equation}\label{eq:2ndorderFDTaylor}
\begin{aligned}
  D_{\rho'}D_\rho v(\ell) &= \sum_{k=2}^{M-1}\sum_{l=1}^{k-1}\frac{1}{l!(k-l)!}\nabla^kv(\ell)[\rho']^l [\rho]^{(k-l)}\\
  &\qquad\qquad+\sum_{l=1}^{M-1}\frac{1}{l!(M-l)!}\nabla^{M}v(\ell+\theta_l)[\rho']^l[\rho]^{(M-l)}.
\end{aligned}
\end{equation}
where $\theta_l$ lies in the convex hull of the set of points $\{0,\rho',\rho,\rho'+\rho\}$ for $l=1,\ldots,M-1$. 

To start, we may use this expansion to approximate the action of $H$ on a smooth function. We note that $H$ was defined by \eqref{eq:Hdefinition} as
\[
  H[v](\ell) = \sum_{\rho, \rho' \in \Rc} 
  \nabla^2 V(0)_{\rho \rho'} D_{-\rho'}D_\rho v(\ell).
\]
The symmetry assumption \eqref{eq:pointsymmetry} implies $\nabla^2 V(0)_{\rho \rho'} = \nabla^2 V(0)_{(-\rho) (- \rho')}$. Therefore all the odd orders cancel and using \eqref{eq:2ndorderFDTaylor} we get that there exist tensors $\mathbb{C}^{2n}\in \R^{N^2\times d^{2n}}$ for $n\geq 1$ such that for any smooth function $v$,
\begin{equation}\label{eq:Hcontexpansion}
  H[v](\ell) = \sum_{n=1}^{M_1}\mathbb{C}^{2n}[\nabla^{2n} v(\ell)]+\mathcal{T}^{(3,1)}\big(\nabla^{2M_1+2} v\big),
\end{equation}
where $\mathcal{T}^{(3,1)}$ collects the Taylor error terms.

In coordinates, these tensors are
\[
  \mathbb{C}^{2n}_{abj_1\dots j_{2n}} = \sum_{\rho,\rho'\in\Rc}\sum_{l=1}^{2n-1}\frac{\big(\nabla^2 V(0)_{\rho' \rho}\big)_{ab}\big((-\rho')^{\otimes l}\otimes \rho^{\otimes(n-l)}\big)_{j_1\dots j_k}}{l!(2n-l)!}
\]
for $a,b=1,\ldots,N$, $j_l=1,\ldots,d$.

In particular, in leading order we find $H[v] \approx c_{\rm vol} H^{\rm C}[v]$, as
\begin{align}
\mathbb{C}^{2}[\nabla^{2} v(\ell)] &=-\sum_{\rho,\rho'\in\Rc} \nabla^2 V(0)_{\rho,\rho'} (\nabla^2v(\ell)[\rho,\rho']) \nonumber \\
&= -c_{\rm vol}  \divo \mathbb{C} [\nabla v(\ell)]\nonumber \\
&= c_{\rm vol} H^{\rm C} [v](\ell). \label{eq:Hhighestorder}
\end{align}
Where we used that 
\begin{equation}\label{eq:C2defn}
  \mathbb{C}_{ijkl} = \nabla^2 W(0)_{ijkl} = \frac{1}{c_{\rm vol}} \sum_{\rho,\rho'\in\Rc} (\nabla^2 V(0)_{\rho\rho'})_{ik} \rho_j \rho'_l,
\end{equation}
based on the definition of the Cauchy-Born energy density $W$ in \eqref{eq:defn_W_cb}.

\paragraph{Construction of $u^{\rm C}_p$.}
We recall that so far, we have Taylor expanded the site potential and substituted a truncated series $\bar{u}_p$ on the right-hand side of \eqref{eq:expansionstep1} to obtain \eqref{eq:expansionstep3}, i.e.
\begin{equation*}
\sum_{i=0}^p H[u_i^{\rm C}] + H[s_p] = -\sum_{k=2}^{\tilde{K}} \frac{1}{k!} \delta^{k+1} \mathcal{E}(0)[\bar{u}_p]^k +\mathcal{T}_{p+1}^{(1)}+ \mathcal{T}_{p+1}^{(2)}.
\end{equation*}
We note that the remaining series on the left and right-hand sides of this equation involve only the action of discrete difference operators on smooth continuum functions, and so we may now perform a discrete-to-continuum approximation to handle these terms.

On the left, we insert \eqref{eq:Hhighestorder} from the discussion of $H^{\rm C}$ as the leading order into the full expansion \eqref{eq:Hcontexpansion} to see that there exist tensors $\mathbb{C}^{n}\in \R^{N^2\times d^n}$ for $n\geq3$ such that for a smooth function $v$,
\begin{equation}\label{eq:leadingA2Cexpansion}
  H[v](\ell) = c_{\rm vol} H^{\rm C}[v](\ell) + \sum_{n=2}^{M_1}\mathbb{C}^{2n}[ \nabla^{2n} v(\ell)]+\mathcal{T}^{(3,1)}\big(\nabla^{2M_1+2} v\big).
\end{equation}

We can perform a similar construction for the higher-order terms in the Taylor expansion on the right-hand side of \eqref{eq:expansionstep3}, so that in general, there exist tensors $\mathbb{C}^{\boldsymbol{j}}$ in the natural tensor space which is isomorph to $\R^{N^{k+1}\times d^n}$ for each $\boldsymbol{j} = (j_1, ..., j_k)$ that satisfies $1 \leq j_m \leq M_k-1$ and $\sum_{m=1}^k j_m = n$ for some $n \geq k+1$. This allows us to approximate $\delta^{k+1} \mathcal{E}(0)[v^{\otimes k}](\ell)$ as
\begin{equation}\label{eq:higherA2Cexpansion}
  \delta^{k+1} \mathcal{E}(0)[v]^k(\ell) = \sum_{\substack{\boldsymbol{j} \colon 1\leq j_m \leq M_k \\\sum_m j_m \geq k+1}} \mathbb{C}^{\boldsymbol{j}}[\nabla^{j_1} v(\ell), \dots, \nabla^{j_k} v(\ell)] +\mathcal{T}^{(3,k)}(v).
\end{equation}
Each of the tensors $\mathbb{C}^{\boldsymbol{j}}$ is a sum of tensor products formed of components of $\nabla^{k+1}V(0)$ and a $n$-fold tensor products of vectors from $\Rc$. The final term
\[\mathcal{T}^{(3,k)}(v)=\mathcal{T}^{(3,k)}\big(\nabla v(\ell),\ldots,\nabla^{M_k}v(\ell), \nabla^{M_k+1}v \big)\]
collects all the terms that include at least one Taylor error for $v$. And overall, after inserting the specific $v$, we can combine these errors to
\[\mathcal{T}^{(3)}_{p+1} := -\sum_{i=0}^p \mathcal{T}^{(3,1)}(u^{\rm C}_i)-\sum_{k=2}^{\tilde{K}} \frac{1}{k!}\mathcal{T}^{(3,k)}(\bar{u}_p).\]

Our aim is now to use the decay estimates of the terms $u^{\rm C}_i$ and $u^{\rm CMP}_i$ to identify and collect all terms in \eqref{eq:higherA2Cexpansion} with the same rate of decay into a forcing term on the right-hand side of \eqref{eq:uCiPDE}, thereby defining the terms $\mathcal{S}_q$ which satisfy the estimate $|\nabla^j\mathcal{S}_q|\lesssim \abs{\ell}^{-d-j-q} \log^{\gamma_q} \abs{\ell}$ for all $j$, some $\gamma_q$, and all $\lvert \ell \rvert$ large enough. Recall from \eqref{eq:uCestimate} and \eqref{eq:uCMPestimate} that for $\lvert \ell \rvert $ large enough
\begin{equation} \label{eq:utildedecay}
  \big|\nabla^j u^{\rm C}_i\big|\lesssim C_j |\ell|^{2-d-j-i}\log^{i}|\ell|\quad\text{and}\quad
  \lvert \nabla^j u_i^{\rm CMP} \rvert \leq C_{j} \lvert\ell\rvert^{2-d-j-i},
\end{equation}
and for notational convenience in the following argument, define
$\tilde{u}_i := u^{\rm C}_i+u^{\rm CMP}_i$ for $i <p$ and $\tilde{u}_p:= u^{\rm C}_p$, which satisfies
\[\big|\nabla^j \tilde{u}_i\big|\leq C_j \abs{\ell}^{2-d-j-i}\log^{i}\abs{\ell}\]
for all $j \in \N_0$ while overall $\sum_{i=0}^p \tilde{u}_i = \bar{u}_p$.

Let us consider multi-indices $\boldsymbol{i}=(i_1, \dots, k)$ for choosing multiple $u_i$ in a nonlinear term and $\boldsymbol{j}=(j_1, \dots, j_k)$ for the number of derivatives on each term. We write $ \lvert \boldsymbol{i} \rvert_1 = \sum_m i_m$ for the sum of the components.

By virtue of estimate \eqref{eq:utildedecay}, and the properties of tensor products, we have
\[
  \Big|\nabla^{j_1}\tilde{u}_{i_1}\otimes \dots \otimes \nabla^{j_k}\tilde{u}_{i_k}\Big|\lesssim
  \abs{\ell}^{(2-d)k-\lvert \boldsymbol{j} \rvert_1 - \lvert \boldsymbol{i} \rvert_1}\log^{\lvert \boldsymbol{i} \rvert_1}\abs{\ell}.
\]

We can therefore identify terms with the decay rate $\abs{\ell}^{-d-q}$ in \eqref{eq:higherA2Cexpansion} by looking only at the indices in
\[\mathcal{I}_{q,k,\boldsymbol{j}} := \{\boldsymbol{i} \in  \{0, \dots, p\}^k \colon \lvert \boldsymbol{i} \rvert_1 = d+ q + (2-d)k-\lvert \boldsymbol{j} \rvert_1 \}.\]
Now that we can precisely identify all these terms we define
\begin{align}
\mathcal{S}_q &:= - \frac{1}{c_{\rm vol}}\sum_{k=2}^{\tilde{K}} \frac{1}{k!} \sum_{\substack{\boldsymbol{j} \colon 1\leq j_i \leq M_k \\ \lvert \boldsymbol{j} \rvert_1 \geq k+1}} \sum_{\boldsymbol{i} \in \mathcal{I}_{q,k, \boldsymbol{j}}} \mathbb{C}^{\boldsymbol{j}}[\nabla^{j_1} \tilde{u}_{i_1}(\ell), \dots, \nabla^{j_k} \tilde{u}_{i_k}(\ell)]\nonumber \\
&\quad - \frac{1}{c_{\rm vol}} \sum_{\substack{4\leq n \leq q+2\\ n\, {\rm even}}}\mathbb{C}^{n}[\nabla^n u^{\rm C}_{2+q-n}(\ell)]. \label{eq:Sqdefn}
\end{align}
We only divide by $c_{\rm vol}$ as we want to factor it out from the resulting equation.

We collect all the remaining orders in
\begin{align*}\mathcal{T}^{(4)}_{p+1} &= \sum_{q \geq p+1} \Big(-\sum_{k=2}^{\tilde{K}} \frac{1}{k!} \sum_{\substack{\boldsymbol{j} \colon 1\leq j_i \leq M_k  \\ \lvert \boldsymbol{j} \rvert_1\geq k+1}} \sum_{\boldsymbol{i} \in \mathcal{I}_{q,k,\boldsymbol{j}}} \mathbb{C}^{\boldsymbol{j}}[\nabla^{j_1} \tilde{u}_{i_1}(\ell), \dots, \nabla^{j_k} \tilde{u}_{i_k}(\ell)]\\
&\quad - \sum_{\substack{3\leq n \leq \min\{q+2, M_1\}\\ n\, {\rm even}}}\mathbb{C}^{n}[\nabla^n u^{\rm C}_{2+q-n}(\ell)]\Big).
\end{align*}
Note that as long as we choose the $M_k$ large enough, the definition of the $S_q$, $q \leq p$, in \eqref{eq:Sqdefn} is independent of their precise choice. Any additional terms from choosing $M_k$ even larger only appear in $\mathcal{T}^{(4)}_{p+1}$.

\begin{lemma}\label{lem:thirderror}
  The following estimates holds:
  \begin{align}
  |\nabla^j\mathcal{S}_q|&\lesssim |\ell|_0^{-d-q-j}\log^{q-1}|\ell|_0 \quad \text{for all}\ 1\leq q \leq p, \label{eq:Sqestimate}\\
  |D^j\mathcal{T}^{(3)}_{p+1}|&\lesssim |\ell|_0^{-d-p-1-j}, \label{eq:T3error}\\
  |D^j\mathcal{T}^{(4)}_{p+1}|&\lesssim |\ell|_0^{-d-p-1-j}\log^{p}|\ell|_0,. \label{eq:T4error}
  \end{align}
  for all $j\in\N_0$.
\end{lemma}
\begin{proof}
The order of decay in \eqref{eq:Sqestimate} is clear by construction, we only need to check the number of logarithmic terms. Indeed, the number of logarithmic terms needed is the maximum of all possible $\sum_m i_m = \lvert \boldsymbol{i} \rvert_1$. For $k \geq 2$ and $n \geq k+1$ derivatives we have $\lvert \boldsymbol{i} \rvert_1 = d+ q + (2-d)k-\lvert \boldsymbol{j} \rvert_1$ which is maximal if $\lvert \boldsymbol{j} \rvert_1=k+1$ and $k=2$, therefore $\lvert \boldsymbol{i} \rvert_1 \leq 1-d+q \leq q-1$ for all terms with $k \geq 2$. For the linear terms we directly see that we get at most $q-1$ logarithmic terms from $u_{q-1}$.

\eqref{eq:T3error} is trivial as we can choose $M_k$ as large as we want.

And \eqref{eq:T4error} is just the same estimate as \eqref{eq:Sqestimate}, where now we look at various $q \geq p+1$. The dominant part being the estimate for $q= p+1$. Having discrete differences instead of gradients does not change anything.
\end{proof}

We note the separation between terms which derive from third- and higher-order derivatives of $\mathcal{E}$ for $k>1$ on the first line of \eqref{eq:Sqdefn}, and those which derive from $H$ on the second line. A distinction arises partly due to the functions involved, and because we retain the leading order approximation $H\approx c_{\rm vol} H^{\rm C}$ on the left-hand side of \eqref{eq:uCiPDE}.

For given $q$, it can be checked that under the constraints on the indices, the maximal index $i$ of any $\tilde{u}_i$ involved in one of the terms in the first sum is $i = q-(d-1)\leq q-1$, and in the second, when $i=q-1$. It follows that
\[
  \mathcal{S}_q = \mathcal{S}_q\big(u^{\rm C}_0,u^{\rm CMP}_0,\ldots,u^{\rm C}_{q-1},u^{\rm CMP}_{q-1}\big),
\]
as we assumed in \eqref{eq:uCiPDE}. Furthermore, as $\tilde{K}$ and the $M_k$ are chosen sufficiently large, the definition of $S_q$, $1 \leq q \leq p$, given here does not depend on the induction step $p$ and thus remains the same over the iteration. So given that we assume $H^{\rm C}[u^{\rm C}_q] = \mathcal{S}_q$, these terms are fixed iteratively. For $q=p$, we can now specify $u^{\rm C}_p$ to be a solution to
\[
  H^{\rm C}[u^{\rm C}_p] = \mathcal{S}_p\big(u^{\rm C}_0,u^{\rm CMP}_0,\ldots,u^{\rm C}_{p-1},u^{\rm CMP}_{p-1}\big),
\]
for $\ell$ sufficiently large, so that these terms cancel on both sides of the equation as well, and we generate a new term in our series expansion of $\bar{u}$. 

Indeed, as $\mathcal{S}_p \in C^\infty(\R^d \backslash B_R(0); \R^N)$ for some radius $R>0$ and satisfies
\[|\nabla^j\mathcal{S}_p|\lesssim |\ell|^{-d-p-j}\log^{p-1}|\ell|\]
there according to \eqref{eq:Sqestimate}, we can use Proposition \ref{prop:contestimates} to see that there exists a $u^{\rm C}_p \in C^\infty(\R^d \backslash B_{2R}(0); \R^N)$ that solves $H^{\rm C}[u^{\rm C}_p] = \mathcal{S}_p$ and satisfies the estimate
\[\lvert \nabla^j u^{\rm C}_p(\ell) \rvert \leq \tilde{C}_j \lvert \ell \rvert^{2-d-p-j} \log^{p} \lvert \ell \rvert\]
for $\lvert \ell \rvert$ large enough. $u^{\rm C}_p$ can of course be extended arbitrarily for small $\lvert \ell \rvert$.

In summary, we may use the definition of $\mathcal{S}_q$ and the remainder terms $\mathcal{T}_{p+1}^{(1)}$, $\mathcal{T}_{p+1}^{(2)}$, $\mathcal{T}_{p+1}^{(3)}$, and $\mathcal{T}_{p+1}^{(4)}$ to rewrite \eqref{eq:expansionstep3} as
\begin{align*}
\sum_{i=0}^p c_{\rm vol} H^{\rm C}[u_i^{\rm C}] + H[s_p] &= \sum_{i=1}^p  c_{\rm vol} \mathcal{S}_i(u_0^{\rm C}, u_0^{\rm CMP}, \dots, u_{i-1}^{\rm C}, u_{i-1}^{\rm CMP})\\
&\quad +\mathcal{T}_{p+1}^{(1)}+ \mathcal{T}_{p+1}^{(2)}+\mathcal{T}_{p+1}^{(3)} + \mathcal{T}_{p+1}^{(4)}.
\end{align*}
In view of the definition of $u^{\rm C}_i$ as a solution to \eqref{eq:uCiPDE}, it follows that this equation reduces to the remainder equation
\begin{equation}\label{eq:spdefn}
  H[s_p] = \mathcal{T}_{p+1}:=\mathcal{T}_{p+1}^{(1)}+ \mathcal{T}_{p+1}^{(2)}+\mathcal{T}_{p+1}^{(3)}+ \mathcal{T}_{p+1}^{(4)},
\end{equation}
which we study in the following section.

\paragraph{Series remainder estimate.}

Note that \eqref{eq:spdefn} is not a linear equation for $s_p$, as $s_p$ also appears on the right hand side in $\mathcal{T}_{p+1}^{(2)}$. In the leading order estimates in \cite{EOS2016} this was resolved through techniques coming from regularity theory which we rely on for $p=0$. For the higher orders, i.e. $p>0$, we can use the (suboptimal) estimates that are available from the previous order. That turns out to be sufficient as the appearance of $s_p$ in $\mathcal{T}_{p+1}^{(2)}$ only occurs in nonlinear terms which give additional decay, see equation \eqref{eq:errorprevordervonly} and Lemma~\ref{lem:seconderror}. 

Combining the error estimates provided by in Lemma~\ref{lem:firsterror}, Lemma~\ref{lem:seconderror} and Lemma~\ref{lem:thirderror}, we have
\[\lvert D^j \mathcal{T}_{p+1} \rvert \lesssim \abs{\ell}^{-d-p-1-j} \log^{p} \abs{\ell}\ \]
for all $j=0, \dots, J-2$. This means we can apply Theorem \ref{thm:structurewithmoments} with $p+1$ instead of $p$ to get
\begin{equation} \label{eq:proof:inductionremainder}
s_p  = \sum_{i=0}^{p} \sum_{k=1}^N b^{(i,k)} : D_{\rm \mathcal{S}}^i \mathcal{G}_k + r_{p+1} 
\end{equation}
and
\[ \lvert D^j r_{p+1} \rvert \lesssim \abs{\ell}^{1-d-p-j}  \log^{p+1} \abs{\ell}\]
for $j=1, \dots, J$, as desired.

\paragraph{Conclusion.} To conclude the induction step we still have to look at that the possible addition of lower order terms from  \eqref{eq:proof:inductionremainder} as they change the multipole terms from one iteration to the next. Instead of discussing under what conditions these terms vanish, we allow for them to be non-zero and consider the more detailed consequences. Indeed, this is fine as long as $u_0^{\rm MP}=0$ remains true as this is part of the results and as long as $u_i^{\rm CMP}$, $i=1,\dots,p-1$, remain unchanged which is important since they are part of the continuum equations \eqref{eq:uCiPDE}. The first just follows from the fact that $\mathcal{I}_0[v]=0$ due to the divergence form of $\mathcal{T}_{p+1}$ combined with sufficient decay. For the second part, let us use the continuum expansion of the multipole terms to write
\[s_p = \sum_{i=1}^{p} v_i^{\rm MP} + r_{p+1} =  \sum_{i=1}^{p} v_i^{\rm CMP} + \tilde{w}^{\rm MP}_{p+1} + r_{p+1}.\]
 We know that $v_i^{\rm CMP}$ is $(2-d-i)$-homogeneous. Combining the decay estimate \eqref{eq:errorprevordervonly} for $s_p$ with the decay estimates for $r_{p+1}$ and $\tilde{w}^{\rm MP}_{p+1}$, we find
\[ \Big\lvert D^j \sum_{i=1}^{p} v_i^{\rm CMP} \Big\rvert \lesssim \lvert \ell \rvert^{2-d-p-j} \log^p \lvert \ell \rvert\]
for $\lvert \ell \rvert$ large enough. This is only compatible with the homogeneity if $v_i^{\rm CMP}=0$ for $i=1,\dots,p-1$. That ends the induction step and thus the proof of Theorem \ref{thm:pointdef} and Theorem \ref{thm:screw}.

As a last step we see that for a point defect, for any $d$, we have $u_0^{\rm C}=0$. We then also find $u_1^{\rm C}=\dots = u_{d}^{\rm C}=0$ as $\mathcal{S}_{i}=0$ for $1 \leq i \leq d$. The first non-trivial terms arises when $u_1^{\rm CMP}$ appears in the nonlinearity, namely
\[\mathcal{S}_{d+1} = \frac{1}{2}\divo \Big(\nabla^3 W (0)\big[\nabla u_1^{\rm CMP}\big]^2\Big).\]
This proves Remark \ref{rem:pointdef}.

\bibliographystyle{myalpha} 

\bibliography{references}

\end{document}